\documentclass[12pt]{article}
\usepackage{geometry}
\usepackage{amsmath}
\usepackage{amsthm,amsfonts,amssymb,amscd}
\usepackage{graphicx}
\usepackage{amsmath,amsfonts,amsthm,amssymb,amscd}
\usepackage{mathtools}
\usepackage{color}  

  \usepackage{paralist}
 \usepackage[colorlinks=true]{hyperref}
  \usepackage{color}

\hypersetup{urlcolor=blue, citecolor=red}

\newcommand\dd{\,\mathrm{d}}

\newcommand\ip[1]{\langle #1\rangle}
\newcommand\bigip[1]{\bigg\langle #1\bigg\rangle}

\newcommand\divv[1]{\mathrm{div}_{#1}}

\newcommand\nd[1]{\frac{\partial #1}{\partial \nu}}

\newcommand\ndy[1]{\frac{\partial #1}{\partial \nu_y}}

\newtheorem{theorem}{Theorem}[section]
\newtheorem{corollary}{Corollary}

\newtheorem{lemma}[theorem]{Lemma}

\newtheorem{assumption}{Assumption}

\theoremstyle{definition}
\newtheorem{definition}[theorem]{Definition}
\newtheorem{remark}{Remark}

\title{A Spectral Target Signature for Thin Surfaces with Higher Order Jump Conditions}
\author{Fioralba Cakoni\thanks{Department of Mathematics, Rutgers University, Piscataway, NJ 08854, USA(fc292@math.rutgers.edu)}, Heejin Lee\thanks{Department of Mathematics, Rutgers University, Piscataway, NJ 08854, USA(hl707@math.rutgers.edu)}, Peter Monk\thanks{Department of Mathematical Sciences, University of Delaware, Newark, DE 19711,  USA(monk@udel.edu)} \\and Yangwen Zhang\thanks{Department of Mathematical Sciences, Carnegie Mellon University, Pittsburgh, PA 15213, USA(yangwenz@andrew.cmu.edu)}}
\date{\empty}

\begin{document}
\maketitle


 \bigskip 
  \centerline{In Memory of Professor Victor Isakov}

\begin{abstract}
\noindent
In this paper we consider the inverse problem of determining structural  properties of a thin anisotropic and dissipative inhomogeneity in ${\mathbb R}^m$, $m=2,3$ from scattering data. In the asymptotic limit as the thickness goes to zero, the thin inhomogeneity is modeled by an open $m-1$ dimensional manifold (here referred to as screen), and the  field inside  is replaced by jump conditions on the total field involving  a second order surface differential operator.  We show that all the  surface coefficients (possibly matrix valued and complex) are uniquely determined from far field patterns of the scattered fields due to infinitely many incident plane waves at a fixed frequency. Then we introduce a target signature  characterized by a novel eigenvalue problem such that the eigenvalues can be determined from measured scattering data, adapting the approach in  \cite{Screens}. Changes in the measured eigenvalues are used to identified changes in the coefficients  without making use of the governing equations that model the healthy screen. In our investigation the shape of the screen  is known, since it represents the object being evaluated. We present  some preliminary numerical  results indicating the validity of our inversion approach.
\end{abstract}

\section{Introduction}

In this paper we are concerned with  nondestructive  evaluation of thin inhomogeneities via  probing with waves. In many contemporary engineering designs one encounters thin structures that are anisotropic, absorbing and dispersive. Inversion methods for fast monitoring of the integrity of such complex structures are highly desirable, and target signatures are suitable for this task. Target signatures are discrete quantities that can be computed from scattering data and used as indicators of changes in the constitutive material properties of the inhomogeneity. Let us first introduce the scattering problem we consider here.  
Let the bounded and connected piecewise smooth region $\mathcal{S} \in {\mathbb R}^{m}$, $m=2,3$ be the support of  a thin inhomogeneity with the constitutive material properties $A$ and $n$. We denote by ${\bf n}$ the unit outward normal vector defined almost everywhere on the boundary $\partial \mathcal{S}$.  Suppose that the incident field and the other fields in the problem are time harmonic, i.e. the time dependent incident field is of the form $\Re\left(u^i(x)e^{i\omega t}\right)$ where $\omega$ is the angular frequency, and where the complex valued spatially dependent function $u^i(x)$ is a solution of 
$$\Delta u^i+k^2 u^i=0\qquad \mbox{in}\;  {\mathbb R}^m.$$
Then the total field $u=u^s+u^i$ in ${\mathbb R}^m\setminus \overline{\mathcal{S}}$, where $u^s$ is the scattered field. If, in addition, $U$  denotes the total field in ${\mathcal S}$  then $u$ and $U$, respectively, satisfy
  \begin{eqnarray}
  \Delta u + k^2 u = 0 \qquad  &\text{\qquad  in \quad}& {\mathbb R}^m\setminus \overline{\mathcal{S}},\label{eq1}\\
  \nabla \cdot  A  \nabla U + k^2nU=0 &\text{\qquad  in \quad}& \mathcal{S}. \label{eq2}
  \end{eqnarray}
Here the wave number $k=\omega/c_{\rm{}ext}$ with $c_{\rm{}ext}$ denoting the wave speed of the homogeneous background. Across the interface the field on either side and their co-normal derivatives are continuous, i.e.  
\begin{equation}
u=U \qquad \mbox{and} \qquad  {\bf n}\cdot \nabla u={\bf n} \cdot A \nabla U \qquad   \mbox{ on } \; \partial \mathcal{S}. \label{eq3}\\
\end{equation}
 Of course the scattered field $u^s$ satisfies the Sommerfeld radiation condition  (see \cite{ColtonKress13})
\begin{equation}\label{SR-c}
\lim_{r\to \infty}{r^{\frac{m-1}{2}}}\left(\frac{\partial u^s}{\partial r}-iku^s\right) =0
\end{equation}
uniformly in $\hat x=x/|x|$, where $x\in{\mathbb R}^m$ and  $r=|x|$. In this paper we consider  plane waves as  incident fields which are  given by  $u^i:=e^{ikx\cdot d}$ where the unit vector $d$ is the incident direction. Instead of plane waves,  it is also possible to consider incident waves due to point sources located outside $\mathcal{S}$, in which case the obvious modifications need to be made in the formulation of the problem.

Now, we assume that $\mathcal{S}$ is cylindrical with maximum thickness  $2\delta>0$ that is bounded above and below by smooth bounded and connected  $m-1$ dimensional manifolds $\Gamma^+$ and $\Gamma^-$. Furthermore there is a  smooth surface $\Gamma$ with a  chosen unit normal  $\nu$ such that  $\Gamma^{\pm}:= {\pm}\delta f^{\pm}(s)\nu(s)$ with $s\in \Gamma$ and $0\leq f^{\pm}\leq 1 $ where $f^\pm$ are smooth profile functions defined on $\Gamma$ with boundary $\partial \Gamma$ (see Figure \ref{conf}).  
\begin{figure}[hh]
\begin{center}
\begin{tabular}{ccc}
\includegraphics[width=6cm,height=4cm]{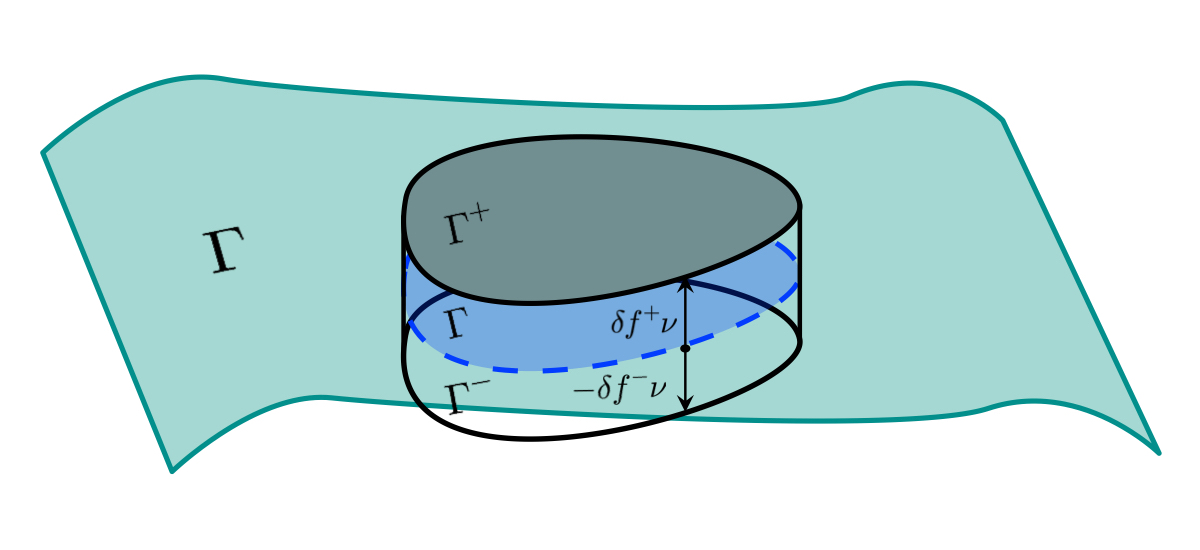} & 
\includegraphics[width=6cm,height=4cm]{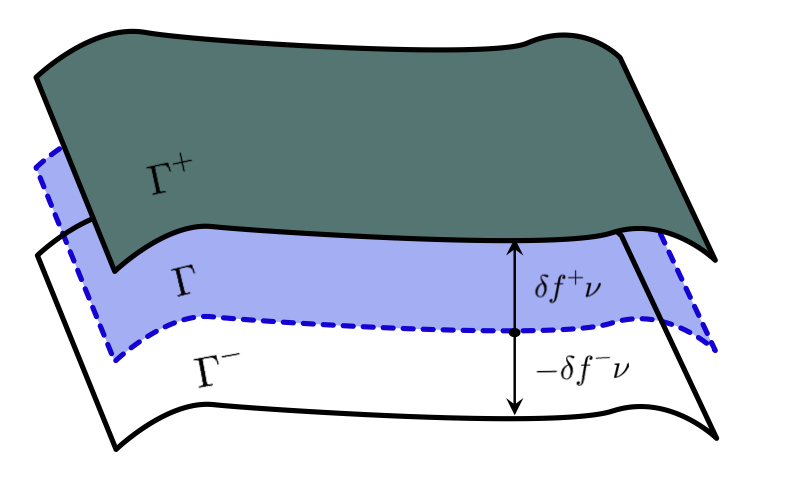} 
\end{tabular}
\caption{\small  }
\label{conf}
\end{center}
\end{figure}
The specific feature of our inhomogeneity is that the thickness $\delta$ is much smaller than the interrogating wavelength in  free space $\lambda=2\pi/k$. This introduces an essential computational difficulty in the numerical solution of the forward problem, and more importantly, for our purpose, in the inverse coefficient problem. 
Our goal is to design a sensitive target signature to detect changes in the coefficients $A$ and $n$ which involve the small scale of the thickness. Hence it is reasonable to use an asymptotic method for small $\delta$ and derive an approximate model where the inhomogeneity $\mathcal{S}$ is reduced to the open surface $\Gamma$. There is a vast literature in asymptotic methods for thin layers  \cite{CakoniTeresaHaddarMonk16}, \cite{electro}, \cite{houssem2} \cite{houssem1}, \cite{HM}, \cite{HM1}, \cite{asym},   where the different ways of performing the asymptotic analysis  lead to particular  types of jump conditions across $\Gamma$. For example  following the asymptotic   approach developed in \cite{irene-thesis} and used in \cite{CakoniTeresaHaddarMonk16} for a different inverse problems,  (\ref{eq2}) is replaced by the following  {\it approximate transmission conditions}
\begin{eqnarray}
 \left[ u \right] &=&  \delta(A^{-1}-1)(f^+ + f^-)  \left<  \frac{\partial u}{\partial \nu} \right>	\text{ on } \Gamma	\label{ac1}\\
 \left[ \frac{\partial u}{\partial \nu} \right] 	&=&   \left(- \nabla_\Gamma \cdot \delta (A-1)(f^+ + f^-)\nabla_\Gamma
  +  \delta k^2(1-n) \right) \left< u\right> \text{ on } \Gamma	\label{ac2}
 \end{eqnarray}
 where  $[u] :=  u^+ - u^-$ and $\ip{u}:= (u^+ + u^-)/2$ where $u^\pm(x) = \lim\limits_{h \rightarrow 0+} u(x\pm h\nu)$ and  for a matrix valued function $\partial u^\pm(x)/\partial\nu = \lim\limits_{h \rightarrow 0+} \nu \cdot \nabla u(x\pm h\nu)$ for $x \in \Gamma$, with  $\nu$   denoting a chosen normal direction to the oriented surface $\Gamma$. One can see that the coefficients in the jump conditions involve the constitutive material properties of the inhomogeneity as well as its thickness. If the inhomogeneity is a cylinder with constant thickness, then $f^+=f^-=1$ and in this case it is possible to consider a matrix valued coefficient $A$ that is  $\Gamma$-orthotropic independent of the normal direction $\nu$. For the convergence analysis of this type of  approximate models we refer the reader to   \cite{houssem2} \cite{houssem1}, \cite{HM}, \cite{HM1}, \cite{asym}. 
 
 \noindent
In this paper we  use a screen model of the above type with anisotropic coefficients on the surface, generalizing the model considered  \cite{Screens} to a more realistic situation. Constitutive material properties of the thin inhomogeneity are represented by a surface matrix  function, and two scalar surface functions. Our inverse problem is to determine information about  these coefficients from a knowledge of the far field pattern of the scattered fields due to infinitely many plane waves at fixed frequency, provide that $\Gamma$ is known. We refer the reader to \cite{crack3}, \cite{elena},  \cite{crack4}, \cite{crack2}, \cite{CC2003},  \cite{KR}, \cite{haddar}, \cite{ZC2009}, for various  reconstruction methods for  the shape of an open surface in inverse scattering. Our inversion method is based on a target signature characterized by a novel eigenvalue problem such that the eigenvalues can be determined from measured scattering data, adapting the approach in  \cite{Screens}. Changes in the measured eigenvalues are used to identify changes in the coefficients  without making use of the governing equations. 
\\
\noindent
Spectral target signatures originated with the Singularity Expansion Method based on  resonances (or scattering poles) \cite{baum_91}. Recently, transmission eigenvalues \cite{tran1}, \cite{CCH-b} (which are related to non-scattering frequencies \cite{vog}, \cite{mikko}) have been 
 successfully used as target signature for non-absorbing  inhomogeneities, since real  transmission eigenvalues that can be determined from multi-frequency scattering data exist only for real valued coefficients, (see e.g. \cite{tran2}, \cite{tra3}).  To deal with absorbing and dispersive media and develop a spectral target signature measurable from single frequency scattering data, a general framework was introduced in \cite{CakoniColtonMengMonk16}   to modify the far field operator whose injectivity leads to novel eigenvalue problems. This idea  was further  developed for inhomogeneities with nonempty interior in  \cite{ACH2017}, \cite{IEEE}, \cite{mod2}, \cite{C2020}, \cite{mod3}, \cite{mod1}. Transmission eigenvalues are used   in \cite{crackH} to characterize the density of  small cracks, whereas new eigenvalues  are derived in  \cite{Screens} and \cite{C2018} as target signature for open surfaces based on an appropriate modification of the far field operator. 

\noindent
In the next section we formulate precisely the  inverse problem, and prove a uniqueness result. In Section 3 we  introduce an appropriate modification of the far field operator leading to a new eigenvalue problem that serve as target signature for the screen. Section 4 is dedicated to the analysis of this eigenvalue problem connecting  eigenvalues to the unknown coefficients, whereas is  Section 5 we show how the eigenvalues are determined from scattering data at a fixed frequency. The last section presents some preliminary numerical experiments.  

\noindent
We finally note that the  fundamental ideas of  uniqueness proof and the employment of point sources in the linear sampling method  relate to the celebrated work by Victor Isakov on  inverse coefficients  problem for hyperbolic  partial differential equations. For his contributions in this area we refer the reader to  the monograph \cite{isakov}  which has become a classic in the theory of inverse problem.

\section{Formulation of the Inverse Problem}\label{forIP}
We start by formulating rigorously  our scattering problem. Let $\Gamma\subset {\mathbb R}^m$, $m=2,3$ be an $m-1$ dimensional smooth compact open manifold with boundary. We further assume that $\Gamma$ is simply connected and non self-intersecting such that  it can be embedded as part of a piece-wise smooth closed boundary $\partial D$ circumscribing a bounded connected region $D\subset {\mathbb R}^m$. This determines two sides of $\Gamma$ and  we choose the positive side determined by the unit normal vector $\nu$ on $\Gamma$ that coincides with the normal direction outward to $D$. The scattering problem is: given $u^i$ find the total field $u^s+u^i$ such that 
\begin{align}\label{totalfield}
\left\{\begin{aligned}\quad
&\Delta u + k^2 u = 0 \quad  \text{in } \mathbb{R}^m \setminus \overline{\Gamma}, \\
 &[u] = \alpha \bigip{\frac{\partial u}{\partial \nu}} \quad  \text{on } \Gamma, \\
& \bigg[ \frac{\partial u}{\partial \nu} \bigg] 
= \Big( - \nabla_\Gamma \cdot \mu \nabla_\Gamma + k^2 \beta \Big) \ip{u} \quad  \text{on }\Gamma,\\
& u=0  \quad  \text{on }\partial \Gamma
     \end{aligned}\right. 
\end{align}
 and $u^s$ satisfies the Sommerfeld radiation condition (\ref{SR-c}). In particular here we consider time harmonic incident plane waves given by  $u^i(x)=e^{ikx\cdot d}$ where the unit vector $d$ denotes the direction of propagation.

\noindent
 For simplicity of presentation, we assume that both $\alpha$ and $\mu$ are real valued functions,  whereas $\beta$ is allowed to be a complex valued function representing absorption. Our discussion can be carried through with obvious modifications if both or either one of the coefficients $\alpha$ and $\mu$ have nonzero imaginary part. In the 2-dimensional case the jump  conditions on $\Gamma$ simply become
$$[u] = \alpha(s) \bigip{\frac{\partial u}{\partial \nu}} \qquad \mbox{and} \qquad  \bigg[ \frac{\partial u}{\partial \nu} \bigg] ={\left(-\frac{\partial }{\partial s}\mu(s) \frac{\partial \left<u\right>}{\partial s}+ k^2\beta(s) \left<u\right>\right)} \qquad \mbox{on}\; \Gamma $$ 
 where $s$ denotes the arc-length  variable on $\Gamma$. In this case all coefficients $\alpha$, $\mu$ and $\beta$ are scalar functions. In the case of ${\mathbb R}^3$ we allow for $\mu$ to be a $2\times 2$ matrix-valued functions defined on $\Gamma$ describing anisotropic thin homogeneities  and $\divv{\Gamma} (\mu \nabla_{\Gamma}u)$ is the scalar anisotropic Laplace-Beltrami operator defined on $\Gamma$. More specifically the tensor  coefficient $\mu$ is function of the inhomogeneities anisotropic  physical  parameters and the geometry of the surface $\Gamma$. Let us precisely define the tensor $\mu$. Obviously, $\mu$ maps  a vector tangential to $\Gamma$ at a point $x\in\Gamma$ to a vector tangential to $\Gamma$ at the same point $x\in\Gamma$. To be more precise, let $\nu(x)$  be the smooth outward  unit normal vector function to $D$ and let  $\hat t_1(x)$ and $\hat t_2(x)$ be two perpendicular vectors on the tangential plane to $\Gamma$ at the point $x$ such that  $\hat t_1, \hat t_2, \nu$ form a right hand coordinative system with origin at  $x$. Then  the matrix $\mu(x)$ is given by the following dyadic expression 
\begin{equation}\label{imp-exp}
\mu(x)=\left(\mu_{11}(x)\hat t_1(x)+\mu_{12}(x)\hat t_2(x)\right)\hat t_1(x)+\left(\mu_{21}(x)\hat t_1(x)+\mu_{22}(x)\hat t_2(x)\right)\hat t_2(x).
\end{equation}
Note that,  if $\xi(x)=a\hat t_1(x)+b \hat t_2(x)$ for some  $a,b\in{\mathbb C}$,  then  $\mu(x)\cdot \xi(x)$ is the tangential vector given by
 $$\mu(x)\cdot \xi(x)=(a \mu_{11}(x)+b\mu_{21}(x))\hat t_1(x)+(a \mu_{12}(x)+b\mu_{22}(x))\hat t_2(x).$$
 From physical considerations we assume that $\mu_{12}(x)=\mu_{21}(x)$ for all $x\in \Gamma$, so that in the case of ${\mathbb R}^3$ we assume that  $\mu$  is a symmetric $2\times 2$ tensor with  entries $\mu_{ij}\in L^\infty(\Gamma)$. The basic assumption throughout the paper is that $\mu$ is uniformly positive definite, i.e.
 \begin{equation}\label{pos}
\mu(x)\geq \mu_0 \; \; \mbox{in } {\mathbb R}^2\qquad \mbox{or} \qquad \overline{\xi(x)}^\top \mu(x)\cdot \xi(x)\geq \mu_0 |\xi(x)|^2 \; \;\mbox{in } {\mathbb R}^3
 \end{equation}
  where $\mu_0>0$ is a positive constant (independent of $x$), and (\ref{pos}) holds for almost every point $x\in \Gamma$ and every vector $\xi\in {\mathbb R}^3$ tangential to $\Gamma$ at $x$. The coefficient $\alpha$ is a $L^\infty(\Gamma)$ real valued function such that $\alpha^{-1} \in L^\infty(\Gamma)$. The coefficient $\beta$ is a complex valued function in $L^{\infty}(\Gamma)$ of the form $\beta(x)=\beta_r+\frac{i}{k}\beta_i$, such that $\Im(\beta)=\beta_i/k\leq 0$ which models absorbing and dispersive properties of the inhomogeneity, 

\noindent
To establish the well-posedeness of the scattering problem (\ref{totalfield}) we define the spaces  
$$V(\mathbb{R}^m\setminus \overline{\Gamma}) := \{ u \in H^1_{loc}(\mathbb{R}^m\setminus \overline{\Gamma}) : \nabla_\Gamma u^\pm \in L^2(\Gamma)\},$$ 
$$V_0(\mathbb{R}^m\setminus \overline{\Gamma}) := \{ u \in H^1_{loc}(\mathbb{R}^m\setminus \overline{\Gamma}) : \nabla_\Gamma u^\pm \in L^2(\Gamma) \: \text{and} \: u|_{\partial \Gamma} = 0\}{.}$$ 
\noindent
Then in \cite{H-thesis} it is shown that there exist a unique solution $u\in V_0(\mathbb{R}^m\setminus \overline{\Gamma})$ of (\ref{totalfield}) which depends continuously on $u^i$ with respect to the norm
\[
\|u \|^2_{V}:= \| u \|^2_{H^1(B_R\setminus \overline{D})} + \| u \|^2_{H^1(D)}+\| \ip{u} \|^2_{H^1(\Gamma)},
\]
for every ball $B_R$ of radius $R>0$ large enough.  Note that the direct scattering problem is a particular case of this problem: Find $w\in V(\mathbb{R}^m\setminus \overline{\Gamma}) $ such that $w+u^i\in V_0(\mathbb{R}^m\setminus \overline{\Gamma})$ such that 
\begin{align}\label{w-field}
\left\{\begin{aligned}\quad
&\Delta w + k^2 w = 0 \quad  \text{in } \mathbb{R}^m \setminus \overline{\Gamma}, \\
 &[w] = \alpha \bigip{\frac{\partial w}{\partial \nu}} +\alpha \varphi\quad  \text{on } \Gamma, \\
& \bigg[ \frac{\partial w}{\partial \nu} \bigg] 
= \Big( - \nabla_\Gamma \cdot \mu \nabla_\Gamma + k^2 \beta \Big) \ip{w}+\psi \quad  \text{on }\Gamma,\\
& \lim_{r\to \infty}{r^{\frac{m-1}{2}}}\left(\frac{\partial w}{\partial r}-ik w \right) =0
     \end{aligned}\right. 
\end{align}
where $\varphi = \ip{\partial u^i/ \partial \nu} |_\Gamma$ and $\psi = {\big( - \nabla_\Gamma \cdot \mu \nabla_\Gamma + k^2 \beta \big) \ip{u^i}|_\Gamma}$.  In general ${\varphi}$ and ${\psi}$ can be 
\begin{equation}\label{trace}\left\{\begin{array}{rrrclll}
& \varphi:=\displaystyle{\left<\frac{\partial v}{\partial \nu} \right>-\frac{1}{\alpha}\left[ v \right]} \hspace*{2cm}& \\
& &\\
& \psi:=\Big( - \nabla_\Gamma \cdot \mu \nabla_\Gamma + k^2 \beta \Big) \ip{v} -\displaystyle{\left[\frac{\partial v}{\partial \nu} \right]}&
\end{array}\right.
\end{equation}
for some $v\in V(\mathbb{R}^m\setminus \overline{\Gamma}) $ with $\Delta v \in L^2(B_R\setminus\Gamma)$ for all $R>0$. In fact in \cite{H-thesis} it is shown that  there is a unique solution $w\in V(\mathbb{R}^m\setminus \overline{\Gamma}) $, with $w-v\in V_0(\mathbb{R}^m\setminus \overline{\Gamma}) $ of (\ref{w-field}). For later use we define the following trace space on $\Gamma$ of functions $u\in V_0(\mathbb{R}^m\setminus \overline{\Gamma}) $ by
\begin{equation}\label{tracesp}
V_0\left(\Gamma \right) :=\left\{u\in H^{1/2}(\Gamma) \;\; \mbox{such that}\;\; \nabla_{\Gamma} u\in L^2(\Gamma) \: \text{and} \: u|_{\partial \Gamma} = 0\right\} 
\end{equation} 
and its dual $V^{-1}\left(\Gamma\right)$ with respect to the following duality pairing
\begin{equation}\label{trace-dual}
\left(u,v\right)_{V_0\left(\Gamma\right),V^{-1}\left(\Gamma\right)}:=\left(u,v\right)_{H^{1/2}(\Gamma),\tilde H^{-1/2}(\Gamma)}+\left(\nabla_\Gamma u, \nabla_{\Gamma} v\right)_{L^2(\Gamma), L^2(\Gamma)}.
\end{equation}
We define $\tilde{H}^{1/2}(\Gamma)$ and $\tilde{H}^{-1/2}(\Gamma)$ consist of functions in ${H}^{1/2}(\Gamma)$ and ${H}^{-1/2}(\Gamma)$ that can be extended by zero to the entire boundary $\partial D$ as $H^{1/2}$ and $H^{-1/2}$ functions, respectively. They are duals of $H^{-1/2}\left(\Gamma\right)$ and $H^{1/2}\left(\Gamma\right)$, respectively. Note that  for $v\in V(\mathbb{R}^m\setminus \overline{\Gamma})$, such that $\Delta v\in L^2({\mathbb R}^m)$ we have that $[v]\in\tilde H^{1/2}(\Gamma)$ and $[\partial v/\partial \nu]\in \tilde H^{-1/2}(\Gamma)$.  Hence $\varphi\in H^{-1/2}(\Gamma)$ and $\psi\in {V^{-1}(\Gamma)}$.

\noindent
For later use we also define
\begin{equation}\label{VD}
V(D)  := \{ u \in H^1(D) :  \nabla_\Gamma u \in L^2(\Gamma)\},
\end{equation}
\begin{equation}\label{VD0}
V_0(D)  := \{ u \in H^1(D) :  \nabla_\Gamma u \in L^2(\Gamma)  \; \mbox{and} \;  u|_{\partial \Gamma} = 0\},
\end{equation}
equipped with the graph norm, and similarly
\begin{equation}\label{VED}
V({\mathbb R}^m\setminus D)  := \{ u \in H_{loc}^1({\mathbb R}^m\setminus D) :  \nabla_\Gamma u \in L^2(\Gamma)\},
\end{equation}
\begin{equation}\label{VED0}
V_0({\mathbb R}^m\setminus D)  := \{ u \in H_{loc}^1({\mathbb R}^m\setminus D) :  \nabla_\Gamma u \in L^2(\Gamma)\; \mbox{and} \;  u|_{\partial \Gamma} = 0\}.
\end{equation}

\bigskip
\noindent
 It is known, thanks to the radiation condition  (\ref{SR-c}), that the scattered field $u^s(x, d)$ due to the plane wave incident field  $u^i(x,d)=e^{ikx\cdot d}$ assumes the asymptotic behavior \cite{ColtonKress13}
\begin{align}\label{asymp}
u^s(x,d) = \frac{e^{ikr}}{r^{\frac{m-1}{2}}}u_\infty(\hat{x}, d) + O(r^{-\frac{m+1}{2}})\qquad \mbox{as} \;\; r=|x|\to \infty
\end{align}
uniformly in all directions $\hat{x}=x/|x|$. The function $u_\infty(\hat x,d)$ defined on the unit sphere $\mathbb{S}^{m-1}:=\left\{x\in {\mathbb R}^m, \; |x|=1\right\}$ is called the $\mathit{far}$-$\mathit{field}$ $\mathit{pattern}$ of the scattered wave.  The far field patter for various incident field are the data we use  solve the inverse scattering problem.
\medskip

\noindent
The (measured) {\it scattering data} is $u_\infty(\hat x,d)$ for all observation directions $\hat x \in \mathbb{S}^{m-1}$ and  all incident directions $d\in \mathbb{S}^{m-1}$. The {\it inverse problem} of  interest to us is:  from the scattering data determine  information about the boundary coefficients $\alpha$, $\mu$ and $\beta$, provided $\Gamma$ is known.

\begin{remark}\label{rem1}  It is reasonable to replace the last Dirichlet condition on $\partial \Gamma$ in (\ref{totalfield})  with a Neumann type condition, i.e.  the normal derivative on $\partial \Gamma$ tangential to $\partial D$ vanishes (see e.g. \cite{electro}). Furthermore, everything  in this paper holds true if the scattering data is given for incident directions $d\in \mathbb{S}^{m-1}_1\subset \mathbb{S}^{m-1}$ and observation direction $\hat x\in \mathbb{S}^{m-1}_2\subset \mathbb{S}^{m-1}$, where $\mathbb{S}^{m-1}_1$ and $\mathbb{S}^{m-1}_2$ are two open subsets (possibly the same)  of the unit sphere $\mathbb{S}^{m-1}$
\end{remark}
\noindent
For later use we define here the fundamental solution of the Helmholtz equation 
\begin{align}\label{fsol}
\Phi(x,z) =
\begin{cases}
 \frac{i}{4} H^{(1)}_0(k|x-z|) \quad \text{ in } {\mathbb R}^2, \\[10pt]
 \displaystyle{\frac{e^{ik|x-z|}}{4\pi |x-z|}} \quad \text{ in } {\mathbb R}^3.
 \end{cases}
\end{align}
where $H_0^{(1)}$ is the Hankel function of the first kind and of order $0$. Note that $\Phi(x,z)$ is an outgoing field i.e. satisfies the Sommerfeld radiation condition.
\subsection{Unique determination of the boundary coefficients}
We show that the scattering data uniquely determine all the coefficients.  In the following discussion we assume that the screen and the coefficients  satisfy the assumptions at the beginning of this section.

\noindent
To prove our uniqueness theorem, we need a density result stated in the  following lemma. 
\begin{lemma}\label{densitylemma}
Let $u(\cdot, d)$ be the solution of (\ref{totalfield}) with $u^i(x) = e^{ikd\cdot x}$ and let $u^s(\cdot, d)$ be the corresponding scattered field. If  for $\varphi \in H^{-1/2}(\Gamma)$ and $\psi\in V^{-1}(\Gamma)$ such that
\begin{align*}
\left(\left[u(\cdot, d)\right], \varphi\right)_{\tilde H^{1/2}(\Gamma), \tilde H^{-1/2}(\Gamma)} +\left(\left<u(\cdot, d)\right>, \psi\right)_{V_0(\Gamma), V^{-1}(\Gamma)}= 0, \quad \forall d \in \mathbb{S}^{m-1},
\end{align*}
then $\varphi = 0$ and $\psi=0$. 
\end{lemma}
\begin{proof}
For any $d \in \mathbb{S}^{m-1}$, assume that 
\begin{equation}\label{zero}
\int_{\Gamma} \left[u(\cdot, d)\right] \varphi\dd S -\int_{\Gamma} \left<u(\cdot, d)\right> \psi \dd S = 0,
\end{equation}
where the integrals are interpreted in the sense of duality with $L^2(\Gamma)$ as the pivot space.  Let $w \in V$ be the solution of (\ref{w-field}) with boundary data this $\varphi$ and $\psi$.  After integrating by parts  and using the transmission conditions across $\Gamma$ for $u(\cdot, d)$ and $w$ along with the fact that both  $u(\cdot, d)$ and $w$ are zero on $\partial \Gamma$,  (\ref{zero})  becomes 
\begin{align*}
0 & = \int_{\Gamma} \left[u(\cdot, d)\right]\left(\frac{1}{\alpha}[w]-  \bigip{\frac{\partial w}{\partial \nu}}\right)\dd S\\
+&\int_{\Gamma} \left<u(\cdot, d)\right> \left(\Big( - \nabla_\Gamma \cdot \mu \nabla_\Gamma + k^2 \beta \Big) \ip{w} - \bigg[ \frac{\partial w}{\partial \nu} \bigg] \right) \dd S\\
=&\int_{\Gamma}\left(  \bigip{\frac{\partial u(\cdot, d)}{\partial \nu}}[w]- \left[u(\cdot, d)\right] \bigip{\frac{\partial w}{\partial \nu}} + \bigg[ \frac{\partial u(\cdot, d)}{\partial \nu} \bigg] \ip{w} - \left<u(\cdot, d)\right>\bigg[ \frac{\partial w}{\partial \nu} \bigg]\right)\dd S\\
{=}& \int_{\partial D} \left( \bigip{\frac{\partial u(\cdot, d)}{\partial \nu}}[w]- \left[u(\cdot, d)\right] \bigip{\frac{\partial w}{\partial \nu}} + \bigg[ \frac{\partial u(\cdot, d)}{\partial \nu} \bigg] \ip{w} - \left<u(\cdot, d)\right>\bigg[ \frac{\partial w}{\partial \nu} \bigg]\right)\dd S,
\end{align*}
where the last equality holds because the jumps of $u(\cdot, d)$ and $w$ and their normal derivatives are zero across $\partial D\setminus{\overline \Gamma}$. Here $D$ is a bounded connected region $D\subset {\mathbb R}^m$ such that  $\Gamma \subset \partial D$. Simplifying the latter leads to 
\begin{align*}
0 &=  \int_{\partial D} \left(\frac{\partial u^+(\cdot, d)}{\partial \nu} w^+ -  u^+(\cdot, d)\frac{\partial w^+}{\partial \nu} \right)\dd S -\int_{\partial D}\left( \frac{\partial u^-(\cdot, d)}{\partial \nu} w^- -  u^-(\cdot, d)\frac{\partial w^-}{\partial \nu} \right)\dd S\\
=&\int_{\partial D} \left(\frac{\partial u^+(\cdot, d)}{\partial \nu} w^+ -  u^+(\cdot, d)\frac{\partial w^+}{\partial \nu} \right)\dd S,
\end{align*}
where the second integral is zero by Green's second identity since both $u$ and $w$ satisfy the Helmholtz equation inside $D$.
Now using that outside $D$ we have $u=u^s+u^i$  and the fact that 
\[
\int_{\partial D} u^s\frac{\partial w^+}{\partial \nu} -  \frac{\partial u^s}{\partial \nu} w^+ \dd S = 0,
\]
for radiating solutions to the Helmholtz equation $u^s$ and $w$, we finally obtain that
\[
 \int_{\partial D} u^i\frac{\partial w^+}{\partial \nu} -  \frac{\partial u^i}{\partial \nu} w^+ \dd S = 0.
\]
Using that $u^i(x) = e^{-ikx \cdot \hat y}$, i.e. the plane wave in the direction $d:=-\hat y$  we have 
\[
 \int_{\partial D} \Phi_\infty(\cdot, \hat y) \frac{\partial {w^+}}{\partial \nu} -  \frac{\partial \Phi_\infty(\cdot, \hat y)}{\partial \nu} {w^+} \dd S = 0,
\]
where $\Phi_\infty(x, \hat y) = \gamma_n e^{-ikx \cdot \hat y}$ with $\gamma_2 = i/4$ and $\gamma_3 = 1/(4\pi)$ is the far field pattern of the fundamental solution $\Phi(y, x)$ as a function of $y$ located at $x$ given by (\ref{fsol}). Since by  Green's representation theorem
$$ w(y):=\int_{\partial D} \Phi(\cdot,y) \frac{\partial {w^+}}{\partial \nu} -  \frac{\partial \Phi(\cdot, y)}{\partial \nu} {w^+} \dd S, \qquad y\in {\mathbb R}^m\setminus{\overline D}$$
(where we have used the symmetry of the fundamental solution),  the above identity {implies that} the far-field pattern $w_\infty(\hat y)$ of $w$ is identically zero for all $\hat y \in \mathbb{S}^{m-1}$. From Rellich's lemma and unique continuation, $w = 0$ in $\mathbb{R}^m \setminus {\Gamma}$ as a function in $V$ and hence $\varphi=0$ and $\psi=0$ from the jump conditions.
\end{proof}

\begin{theorem}\label{coef}
Assume that  for a fixed wave number $k$, the far-field patterns $u_{1, \infty}(\hat{x}, d)$ and $u_{2, \infty}(\hat{x}, d)$ corresponding to  $(\Gamma, \alpha_1, \beta_1, \mu_1)$ and  $(\Gamma, \alpha_2, \beta_2, \mu_2)$  respectively, coincide for all $\hat{x}, d\in \mathbb{S}^{m-1}$, and in addition $\beta_j\in C(\Gamma)$ and $\mu_j\in C^1(\Gamma)$, $j=1,2$. Then $\alpha_1=\alpha_2$,  $\beta_1=\beta_2$ and $\mu_1=  \mu_2$. 
\end{theorem}
\begin{proof} Since the far field patters coincide, from Rellich's lemma  we have that the total fields $u(\cdot, d) := u_1(\cdot, d) = u_2(\cdot, d)$ in $\mathbb{R}^m \setminus \overline{\Gamma}$. From the jump conditions of  $u_1(\cdot,d)$ and $u_2(\cdot, d)$ across $\Gamma$ we have that 
\begin{equation}\label{eqn1}
 \left(\frac{1}{\alpha_1}-\frac{1}{\alpha_2}\right) [u] = 0 \quad \text{ on } \Gamma,
\end{equation}
\begin{equation}\label{eqn2}
\Big(- \nabla_\Gamma \cdot (\mu_1-\mu_2) \nabla_\Gamma + k^2 (\beta_1-\beta_2) \Big) \ip{u}  = 0 \quad \text{ on } \Gamma.
\end{equation}
Note that $u\in V({\mathbb R}^m\setminus \overline{\Gamma})$. Thus for any test functions $\psi \in H^{-1/2}(\Gamma)$ and $\phi \in {V_0(\Gamma)}$ after integrating by parts on the manifold $\Gamma$  we have
\begin{eqnarray*}
0&=&\left([u(\cdot, d)], \left(\frac{1}{\alpha_1}-\frac{1}{\alpha_2}\right) \psi\right)_{\tilde H^{1/2}(\Gamma), H^{-1/2}(\Gamma)}\\
&+&\left(\left<u(\cdot, d)\right>, -{\nabla_{\Gamma} \cdot}(\mu_1-\mu_2) \nabla_{\Gamma} \phi + {k^2}(\beta_1 - \beta_2)\phi\right)_{{V_0(\Gamma)}, V^{-1}(\Gamma)}
\end{eqnarray*}
\noindent
Applying Lemma \ref{densitylemma} we conclude that  
\begin{equation}\label{h0}
\left(\frac{1}{\alpha_1}-\frac{1}{\alpha_2}\right) \psi=0,  \quad \forall \psi \in H^{-1/2}(\Gamma).
\end{equation}
\begin{equation}\label{h1}
-{\nabla_{\Gamma} \cdot}(\mu_1-\mu_2) \nabla_{\Gamma} \phi + {k^2}(\beta_1 - \beta_2)\phi = 0, \quad \forall \phi \in {V_0(\Gamma)}.
\end{equation}
By taking $\psi=1$ {in (\ref{h0})}  we obtain that $\alpha_1=\alpha_2$ on $\Gamma$. Next consider (\ref{h1}) and for {any} $y_0\in \Gamma$, let $B_{\rho}(y_0)$ be an open ball centered at $y_0$ of radius $\rho>0$ such that $B_{\rho}(y_0) \cap \Gamma$ is contained in the interior of $\Gamma$. Consider a smooth function $\phi \in C^\infty(\Gamma)$ such that $\phi = 1$ in $B_{\rho/4}(y_0) \cap \Gamma$ and $\phi = 0$ in  $\Gamma \setminus \overline{B_{\rho/2}(y_0)}$. Here, $B_{\rho/4}(y_0)$ and $B_{\rho/2}(y_0)$ are the balls centered at $y_0$ of radius, respectively, $\rho/4$ and $\rho/2$ so we have $B_{\rho/4}(y_0) \subset B_{\rho/2}(y_0) \subset B_{\rho}(y_0)$. Then, since now (\ref{h1}) holds point-wise we can conclude that 
$$
\beta_1 - \beta_2 = 0 \quad \text{ in } B_{\rho/8}\cap {\Gamma}.
$$
Since $\beta_1$ and $\beta_2$ are continuous on $\Gamma$ and $y_0 \in B_{\rho/8}\cap {\Gamma}$, we have $\beta_1(y_0) = \beta_2(y_0)$. Since $y_0$ is arbitrary, we conclude that the latter holds for all $y_0 \in \Gamma$. Therefore, $\beta_1 = \beta_2$ on $\Gamma$. Finally letting $\beta_1 - \beta_2=0$ in (\ref{h1}), multiplying by $\phi \in {V_0(\Gamma)}$ and integrating by parts we obtain
\begin{equation}\label{matrix}
0=-\int_{\Gamma} \nabla_{\Gamma} \big((\mu_1-\mu_2) \nabla_{\Gamma} \phi \big) {\phi} \dd S =  \int_{\Gamma} (\nabla_{\Gamma} {\phi}) \cdot (\mu_1-\mu_2) (\nabla_{\Gamma} \phi)  \dd S,
\end{equation}
for all $\phi \in {V_0(\Gamma)}$. If $\mu_1(x_0) \neq \mu_2(x_0)$ for some $x_0 \in \Gamma$, without loss of generality, $\mu_1(x_0) - \mu_2(x_0) \geq {\mu}_0 >0$ for some constant ${\mu}_0$ (where in the case of ${\mathbb R}^3$ the inequality is understood in terms of positive definite tensors). By continuity  there exists $\epsilon >0$ such that $\mu_1(x) - \mu_2(x)>0$ for all $x \in \Gamma\cap B_{\epsilon}(x_0)$, where $B_{\epsilon}(x_0)$ is the ball of radius $\epsilon$ centered at $x_0$. We choose a smooth function $\phi$ compactly supported in $\Gamma \cap B_{\epsilon}(x_0)$ in (\ref{matrix}). The positivity of   $\mu_1 - \mu_2$ implies  that $\nabla_{\Gamma} \phi= 0$ on $\Gamma\cap B_{\epsilon}(x_0)$, hence $\phi$ is a constant on $\Gamma \cap B_{\epsilon}(x_0)$. However, this contradicts  the fact that $\phi$ is compactly supported in $\Gamma \cap B_{\epsilon}(x_0)$. Therefore, $\mu_1 = \mu_2$ on $\Gamma$.
 \end{proof}
\noindent
Instead of reconstructing the coefficients based  on optimization  techniques, in the following we propose and analyze a spectral target signature measurable from the scattering data which identifies changes in the constitutive material properties  of the screen.  We remark that this target signature can detect such changes without knowing the base ``healthy" value the coefficients nor reconstructing them.
\section{The Modified Far Field Operator and a Related Eigenvalue Problem}
The scattering data defines  the {\it far field operator} $F: L^2(\mathbb{S}^{m-1}) \rightarrow L^2(\mathbb{S}^{m-1})$ by
\begin{align*}
Fg (\hat{x}) = \int_{\mathbb{S}^{m-1}} u_\infty(\hat{x}, d) g(d) \dd S(d).
\end{align*}
The injectivity of this operator is related to the geometry of $\Gamma$, in particular $F$ is injective with dense range if and only if there is {\it {no Herglotz wave function}} given by
\begin{equation}\label{hf}
u^i_g(x):=\int_{\mathbb{S}^{m-1}}g(d)e^{ikx\cdot d} \dd S(d), \qquad g\in L^2(\mathbb{S}^{m-1})
\end{equation}
such that $\ip{\partial u^i_g/ \partial \nu} |_\Gamma=0$ and $ {\big( - \nabla_\Gamma \cdot \mu \nabla_\Gamma + k^2 \beta \big)} \ip{u^i_g}|_\Gamma=0$ \cite{CC2003}, \cite{H-thesis}. Hence to introduce the eigenvalue problem we need the following auxiliary scattering problem. Let $D$ {be} a bounded connected region $D\subset {\mathbb R}^m$ such that  $\Gamma \subset \partial D$ and  $\lambda \in \mathbb{C}$ with $\Im (\lambda) \geq 0$, then  find $h^{(\lambda)}$ such that
\begin{align}\label{auxprob}
\left\{\begin{aligned}\quad
&\Delta h^{(\lambda)} + k^2 h^{(\lambda)} = 0 \quad  \text{in } \mathbb{R}^m \setminus \overline{D}, \\
& h^{(\lambda)} = h^{(\lambda),s} + u^i \\
 &h^{(\lambda)}=0 \quad  \text{on } \Gamma, \\
& \frac{\partial h^{(\lambda)}}{\partial \nu} + \lambda h^{(\lambda)} = 0 \quad  \text{on } \partial D \setminus \Gamma,
     \end{aligned}\right.
\end{align}
where $h^{(\lambda), s}$ is the scattered field and $u^i = e^{ik d \cdot x}$ and $h^{(\lambda),s}$ satisfies the Sommerfeld radiation radiation condition  (\ref{SR-c}). Like above, $h^{(\lambda),s}(x,d)$  satisfies  the asymptotic behavior (\ref{asymp}) with $h^{(\lambda)}_\infty(\hat{x}, d)$ the corresponding far-field pattern. The latter defines the corresponding far field operator $F^{(\lambda)}: L^2(\mathbb{S}^{m-1}) \rightarrow L^2(\mathbb{S}^{m-1})$ by
\begin{align*}
F^{(\lambda)}g (\hat{x}) = \int_{\mathbb{S}^{m-1}} h^{(\lambda)}_\infty(\hat{x}, d)  g(d) \dd S(d).
\end{align*}
Now we consider the {\it{modified far field operator}} $\mathcal{F}: L^2(\mathbb{S}^{m-1}) \rightarrow L^2(\mathbb{S}^{m-1})$ 
\begin{align}\label{modifiedffo}
\mathcal{F} g (\hat{x}):= (Fg -F^{(\lambda)}g) (\hat{x}) =\int_{\mathbb{S}^{m-1}} \big( u_\infty(\hat{x};d) - h^{(\lambda)}_\infty(\hat{x};d) \big) g(d) \dd S(d).
\end{align}
Note that for the purpose of the inverse problem we assume that $\mathcal{F} g$ is known since $Fg$ is known from measurements whereas $F^{(\lambda)}$ can be precomputed since the problems involves only $\Gamma$ which we assume is known. Now we ask the question whether the modified far field operator is injective. Indeed, assume that $\mathcal{F}g = 0$ for some $g \in L^2(\mathbb{S}^{m-1})$ and  let $u^i_g$ be the Herglotz wave function (\ref{hf}) with {kernel $g$}. Define solutions $u_g := u_g^s + u_g^i$ and $h_g^{(\lambda)} := h_g^{(\lambda), s}+ u_g^i$ of the scattering problems \eqref{totalfield} and \eqref{auxprob}, respectively, with their far-field patterns $u_{g, \infty}$ and $h^{(\lambda)}_{g, \infty}$.   By linearity, since $\mathcal{F}g = u_{g, \infty} - h^{(\lambda)}_{g, \infty}$, we have  that $\mathcal{F}g = 0$ implies that $u_{g, \infty} = h^{(\lambda)}_{g, \infty}$ on $\mathbb{S}^{m-1}$. From Rellich's Lemma, we have that $u_g^s = h_g^{(\lambda),s}$  in $\mathbb{R}^m \setminus D$. Then, $u_g^+ = h_g^{(\lambda)}$ in $\mathbb{R}^m \setminus \overline{D} $ where $u_g^+ = u_g|_{\mathbb{R}^m \setminus \overline{D}}$. From the boundary conditions in \eqref{auxprob}, $u_g^+ = 0$ on $\Gamma$ and  $\frac{\partial u_g^+}{\partial \nu} + \lambda u_g^+ = 0$ on $ \partial D \setminus \Gamma$. Then, from the boundary conditions in \eqref{totalfield}, $u_g^- = u_g|_D$ satisfies the following:
\begin{align}\label{egprob}
\left\{\begin{aligned}\quad
&\Delta h + k^2 h = 0 \quad  \text{in } D, \\
& \frac{\partial h}{\partial \nu} = -{\frac{1}{4}}\Big( - \nabla_\Gamma \cdot {\mu}\nabla_\Gamma + k^2 \beta +{\frac{4}{\alpha}} \Big) h \quad  \text{on }\Gamma, \\
 & \frac{\partial h}{\partial \nu} + \lambda h =0 \quad  \text{on } \partial D \setminus \Gamma,\\
&h=0 \quad  \text{on } \partial \Gamma
     \end{aligned}\right.
\end{align}
where $h:=u_g^-$.  Thus, if $\lambda$ is not an eigenvalue of \eqref{egprob}, then $\partial h / \partial \nu = h = 0$  on $\partial D \setminus \Gamma$. By Holmgren's Theorem, $u_g^-:=h \equiv 0$ in $D$ and hence by unique continuation $u_g = 0$ in $\mathbb{R}^m \setminus \Gamma$. In addition from the condition on $\Gamma$ in  \eqref{totalfield} we obtain that both jumps of $u_g$ and ${\partial u_g/\partial \nu}$ are zero. Which means that $u_g$ satisfies the Helmholtz equation in ${\mathbb R}^m$ and $u_g^s=-u_g^i$.  Since $u_g^s$ is a radiating solution while $u_g^i$ is not, $g$ must be zero. Therefore, we have proved the following lemma.
\begin{lemma}\label{F_injective}
Assume that $\lambda \in \mathbb{C}$ with $\Im(\lambda)\geq 0$ is not an eigenvalue of \eqref{egprob}. Then, the modified far-field operator $\mathcal{F}: L^2(\mathbb{S}^{m-1}) \rightarrow L^2(\mathbb{S}^{m-1})$ is injective.
\end{lemma}

\begin{remark} From Lemma \ref{F_injective} we know  that   if ${\mathcal F} g=0$  {has} a non-trivial solution, then $\lambda$  is an eigenvalue of \eqref{egprob}.  Note that the converse is not necessarily true, i.e. if $\lambda$ is an eigenvalue of \eqref{egprob}, this doesn't mean that ${\mathcal F}$ is not injective, which will become clear in the following section. Nevertheless the above connection between the modified far field operator ${\mathcal F}$ and the eigenvalue problem  \eqref{egprob} can be exploited to  detect these eigenvalues  from the scattering data.  		
\end{remark}
In the same way as in Lemma 3 in \cite{Screens} one can also show that if  $\lambda$ is not an eigenvalue of  \eqref{egprob}, then the range of ${\mathcal{F}}:L^2(\mathbb{S})\to L^2(\mathbb{S})$ is dense. This is needed when applying the linear sampling method to determine the eigenvalues of \eqref{egprob} from a knowledge of ${\mathcal F}$. Now we are in a position to define precisely the target signatures considered in this paper:

\begin{definition}[Target signatures for the screen $\Gamma$] Given a screen $\Gamma$ and a domain $D$ with $\Gamma\subset\partial D$ the target signature for $\Gamma$ is the set of  eigenvalues of \eqref{egprob}.
\end{definition}

\noindent 
We show in this paper that the target signature for the screen $\Gamma$ is determined from the measured far field data, and hence it can be used to identify changes in the coefficients $\alpha, \beta, \mu$ without reconstructing them.
\section{The Analysis of the Eigenvalue Problem}\label{analEP}
We proceed with the analysis of the eigenvalue problem \eqref{egprob}. In particular we show that its spectrum is discrete, for real value $\beta$ there exist infinitely many real  eigenvalues, and provide relations between the eigenvalues and coefficients $\alpha, \beta, \mu$.

\noindent
If $h \in V_0(D)$ is a nonzero solution to \eqref{egprob}, then $h$ and $\lambda$ satisfy
\begin{multline}\label{variationh}
\int_D  \big( \nabla h \cdot \nabla \overline{\varphi} - k^2 h \overline{\varphi} \big)  \dd x 
+ \frac{1}{{4}}\int_\Gamma \Big( - \nabla_\Gamma \cdot {\mu} \nabla_\Gamma + k^2 \beta + \frac{{4}}{\alpha} \Big)h \overline{\varphi}  \dd S \\
= - \lambda \int_{\partial D \setminus \Gamma} h \overline{\varphi} \dd S, \quad \forall \varphi \in V_0{(D)}.
\end{multline}
\begin{lemma}
Assume that ${\Im({\beta})} \leq 0$. If ${\Im({\lambda})} < 0$ then \eqref{egprob} has only the trivial solution in $V_0(D)$.
\end{lemma}
\begin{proof}
From \eqref{variationh}, we have
\[
\int_D  |\nabla h|^2  - k^2 |h|^2   \dd x 
+ \frac{1}{{4}}\int_\Gamma \mu |\nabla_\Gamma h|^2 +  \Big( k^2 \beta + \frac{{{4}}}{\alpha} \Big)|h|^2  \dd S
= - \lambda \int_{\partial D \setminus \Gamma} |h|^2\dd S.
\]
Taking the imaginary part,
\[
 \frac{k^2}{{4}}\int_\Gamma {\Im({\beta})} |h|^2  \dd S
= - \Im({\lambda})  \int_{\partial D \setminus \Gamma} |h|^2\dd S.
\]
Since $\Im({\beta}) \leq 0$ and $\Im({\lambda}) < 0$, $h =  0$ on $\Gamma$ and $\partial h/ \partial \nu = 0$ on $\Gamma$. Then Holmgren's Theorem implies that $h$ is identically zero in $D$.
\end{proof}

\begin{corollary} All the eigenvalues $\lambda\in {\mathbb C}$ of  \eqref{egprob} satisfy $\Im(\lambda)\geq 0$.
\end{corollary}

\noindent
We can rewrite \eqref{variationh} as $(A+B+\lambda K)(h,\varphi)= 0$, where the sesquilinear forms  $A$, $B$, $K$ from $V_0(D)\times V_0(D)$ to ${\mathbb C}$ 
are  defined  by
\begin{align*}
A(h, \varphi)_V:=
\int_D  \nabla h \cdot \nabla \overline{\varphi} + h \overline{\varphi}   \dd x 
+ \frac{1}{{4}}\int_\Gamma \Big( \mu  \nabla_\Gamma h \cdot \nabla_\Gamma{\overline{\varphi}} + {\mu_0} h \overline{\varphi} \Big) \dd S,
\end{align*}
with some constant $\alpha_0>0$,
\begin{align*}
B(h, \varphi)_V:=
-(k^2+1)\int_D  h \overline{\varphi}   \dd x 
+ \frac{1}{{4}}\int_\Gamma\Big( k^2\beta + \frac{{4}}{\alpha} - {\mu_0} \Big) h\overline{\varphi}  \dd S,
\end{align*}
and \begin{align*}
K(h, \varphi)_V:=
\int_{\partial D \setminus \Gamma} h\overline{\varphi} \dd S.
\end{align*}
By means of the Riesz representation theorem, we define the bounded linear operators $\mathbb{A}$,  $\mathbb{B}$ and  $\mathbb{K}$ on ${V_0(D)}$  by
\begin{eqnarray}
&(\mathbb{A} h, \varphi )_{{V_0(D)}} = A(h, \varphi), \qquad  (\mathbb{B} h, \varphi )_{{V_0(D)}} = B(h, \varphi), & \nonumber\\
&\mbox{and} \quad (\mathbb{K} h, \varphi )_{{V_0(D)}} = K(h, \varphi),  \quad \forall \, h, \varphi \in V_0(D)& \label{opdef}
\end{eqnarray}
where the scalar product is given by $(\cdot, \cdot)_{{V_0(D)}} := (\cdot, \cdot)_{H^1(D)} +(\cdot, \cdot)_{H^1(\Gamma)}$. 
Recall that $\alpha, {\beta,} \mu \in L^\infty(\Gamma)$ , that $\mu$ (possible a tensor) satisfies (\ref{pos}) with some constant $\mu_0$ and that  $\Im({\beta}) \leq 0$. Then, the boundedness of $\mathbb{A}, \mathbb{B}$ and $\mathbb{K}$ follows. For any $h \in {V_0(D)}$,
\begin{align*}
A(h,h) & = \int_D | \nabla h|^2 + |h|^2    \dd x 
+ \frac{1}{{4}}\int_\Gamma \Big( \mu | \nabla_\Gamma h|^2 +{ \mu_0} |h|^2 \Big) \dd S \\
& \geq \|h\|^2_{H^1(D)} + {\frac{\mu_0}{4}} \|h\|^2_{H^1(\Gamma)} \geq {\min(1, \mu_0/4)} \|h\|^2_{{V_0(D)}},
\end{align*}
which shows the coercivity of $A$. Therefore, the operator $\mathbb{A}$ is invertible with bounded inverse.
\noindent
Next, we will show that $\mathbb{B}$ and $\mathbb{K}$ are compact operators. For any $h \in V_0(D)$,
\begin{eqnarray*}
\| \mathbb{B} h\|^2_{{V_0(D)}} 
&=& -(k^2+1) \int_D h \overline{\mathbb{B} h} \dd x
+ \frac{1}{{4}} \int_\Gamma (k^2 \beta +\frac{{4}}{\alpha} - {\mu_0} ) h \overline{\mathbb{B} h}  \dd S \\
& \leq &C \big( \| h\|_{L^2(D)}\| \mathbb{B}h\|_{L^2(D)} + \| h\|_{L^2(\Gamma)}\| \mathbb{B} h\|_{L^2(\Gamma)} \big) \\
& \leq& C \big( \| h\|_{L^2(D)}+ \| h\|_{L^2(\Gamma)} \big) \| \mathbb{B} h\|_{{V_0(D)}} 
\end{eqnarray*}
for some constant $C>0$.
Thus,
\[
 \| \mathbb{B} h\|_{{V_0(D)}}  \leq C\big( \| h\|_{L^2(D)}+ \| h\|_{L^2(\Gamma)} \big).
 \]
Then the  compactness of $\mathbb{B}$ follows from the fact that $H^1(D)$ and $H^1(\Gamma)$ are compactly embedded in $L^2(D)$ and $L^2(\Gamma)$, respectively. Similarly, the compactness of $\mathbb{K}$ follows from
\[
\|\mathbb{K} h \|_{{V_0(D)}} \leq C\| h\|_{L^2(D)}
\]
for some constant $C>0$.

\noindent
Finally the Analytic Fredholm Theory \cite{ColtonKress13} applied to $I +  \mathbb{A}^{-1}( \mathbb{B} + \lambda \mathbb{K}) $ implies that the set of eigenvalues $\lambda\in {\mathbb C}$ is discrete with $\infty$ as the only possible  accumulation point.

\subsection{Relations between eigenvalues and the surface parameters}
We would like to understand how the eigenvalues of the eigenvalue problem \eqref{egprob} relate to the known coefficients $\alpha$, $\mu$ and $\beta$, which satisfy the assumptions in  Section \ref{forIP}. Let us fix a $\tau>0$ such that $k^2$ is not an eigenvalue of the mixed Dirichlet-Generalized Impedance eigenvalue problem of finding $h\in V_0(D)$
\begin{align*}
\left\{\begin{aligned}\quad
&\Delta h + k^2 h = 0 \quad  \text{in } D, \\
& \frac{\partial h}{\partial \nu} = -\frac{1}{{4}}\Big( - \nabla_\Gamma \cdot {\mu}\nabla_\Gamma + k^2 \beta +\frac{{4}}{\alpha} \Big) h \quad  \text{on }\Gamma, \\
 & \frac{\partial h}{\partial \nu} + \tau h =0 \quad  \text{on } \partial D \setminus \Gamma.
\end{aligned}\right.
\end{align*}
For a given wave number $k$, we can always find such a $\tau$ because from the above Fredholm property of \eqref{egprob} this problem can have a nontrivial solution only for a discrete set of the parameter $\lambda$. The choice of $\tau$ guarantees that the operator $(\mathbb{A} + \mathbb{B} + \tau \mathbb{K}): V_0(D)\to V_0(D)$ is invertible,  where the operators $\mathbb{A}, \mathbb{B}$ and $\mathbb{K}$ are defined by \eqref{opdef}.
Therefore we can define the operator $\mathcal{R} : L^2(\partial D \setminus \Gamma) \rightarrow L^2(\partial D \setminus \Gamma)$ that maps a function $\theta \in L^2(\partial D \setminus \Gamma)$ into $h_\theta |_{\partial D\setminus \Gamma}$ where  $h_\theta$, is the unique solution
\[
\left((\mathbb{A} + \mathbb{B} + \tau \mathbb{K})h_\theta, \varphi\right)  = \int_{\partial D \setminus \Gamma} \theta  \overline{\varphi} \dd S.
\]
Since $h_\theta|_{\partial D \setminus \Gamma} \in H^{\frac{1}{2}}(\partial D \setminus \Gamma)$ and $H^{\frac{1}{2}}(\partial D \setminus \Gamma)$ is compactly embedded into $L^2(\partial D \setminus \Gamma)$, the operator $\mathcal{R}$ is compact. Such $h_\theta$ exists from the choice of $\tau$. Then, we see  that $\lambda \in {\mathbb C}$ is an eigenvalue of \eqref{egprob} if and only if
\begin{equation}\label{rf}
(-\lambda + \tau) \mathcal{R} \theta = \theta
\end{equation}
for some nonzero $\theta$. In other words, $\frac{1}{-\lambda + \tau}$  is an eigenvalue of the compact operator $\mathcal{R}$. In particular if we assume that $\Im(\beta)=0$, then the operator $\mathcal{R}$ is self-adjoint, hence for a fixed $\tau$, the eigenvalues $\frac{1}{-\lambda_j + \tau}$ of the operator $\mathcal{R}$ are all real and accumulate to $0$ as $j \rightarrow \infty$. Hence we conclude that in this case all eigenvalues of  \eqref{egprob} are real, and there exists an infinite sequence of $\{\lambda_j\}_{j\geq 1}$ of  real eigenvalues that accumulate to $\pm  \infty$ as $j \rightarrow \infty$ (because the operator ${\mathcal R}$ is not sign definite the eigenvalues may in principle accumulate to both $+\infty$ and $-\infty$). However, in the next theorem we show that the eigenvalues accumulate only to $-\infty$. In addition the corresponding eigenfunctions form a Riesz basis for $V_0(D)$.
\begin{remark} 
If $\Im(\beta)>0$ the eigenvalue problem \eqref{egprob} is non-selfadjoint and in this case all eigenvalues are complex with $\Im(\lambda)>0$. Then, using the theory of Agmon on non self-adjoint eigenvalue problem {in} \cite{Agmon} is possible to prove in a similar way as in \cite{CakoniColtonMengMonk16}  that for smooth coefficients there exits an infinite set of  complex eigenvalues in the upper half complex plane asymptotically  approaching the negative real axis. One could handle the case of existence of eigenvalues for complex coefficients by modifying the impedance condition  on $\partial D\setminus \Gamma$  in the auxiliary and consequently in the eigenvalue problem by introducing a smoothing boundary operator along the lines of the ideas in \cite{C2020} which makes the non self-adjoint operator ${\mathcal R}$ a trace class operator. This idea is considered in \cite{H-thesis}
\end{remark}

\begin{theorem}
If $k^2$ is not an eigenvalue of
\begin{eqnarray}
&\Delta u + k^2 u= 0 \:\:  \text{in }D& \label{evprob1}\\
& \displaystyle{\frac{\partial u}{\partial \nu}} = -\frac{1}{{4}}\Big( - \nabla_\Gamma \cdot \mu\nabla_\Gamma + k^2 \beta + \frac{{4}}{\alpha} \Big) u \:\: \text{on }\Gamma, \quad 
 u=0 \:\:  \text{on } \partial D \setminus \Gamma& \nonumber
\end{eqnarray}
 then there are at most finitely many positive eigenvalues $\lambda$ of \eqref{egprob}.
\end{theorem}

\begin{proof}
Assume to the contrary that there exists a sequence of positive eigenvalues $\lambda_j > 0$ such that $\lambda_j \rightarrow \infty$ as $j \rightarrow \infty$ with normalized eigenfunctions $h_j$ satisfying
\begin{align}\label{normalizedsq}
\|h_j \|_{H^1(D)} + \|h_j\|_{H^1(\Gamma)} = 1.
\end{align}
From \eqref{variationh},
\begin{multline}\label{egf}
\int_D |\nabla h_j |^2 - k^2 |h_j |^2  \dd x 
+ \frac{1}{{4}}\int_\Gamma  \mu |\nabla_\Gamma h_j |^2 + \Big( k^2 \beta + \frac{{4}}{\alpha} \Big)|h_j |^2  \dd S \\
= \int_{\partial D \setminus \Gamma}(- \lambda_j )|h_j|^2  \dd S.
\end{multline}
Since the left-hand side is bounded and $\lambda_j \rightarrow \infty$, $h_j \rightarrow 0$ in $L^2(\partial D \setminus \Gamma)$. Then, up to a subsequence, $h_j(x) \rightarrow 0$ for almost all  $x \in \partial D \setminus \Gamma$. By the assumption \eqref{normalizedsq}, there exists a subsequence, still denoted by $\{h_j\}_{j\in {\mathbb N}}$, that converges weakly in $V_0(D)$ to some $h \in V_0(D)$. In particular $h=0$ in $D \setminus \Gamma$. Furthermore, since each $h_j$ satisfies \eqref{variationh}
\begin{multline*}
\int_D  \big( \nabla h_j \cdot \nabla \overline{\varphi} - k^2 h_j \overline{\varphi} \big)  \dd x 
+ \frac{1}{{4}}\int_\Gamma \mu \nabla_\Gamma h_j \cdot \nabla_\Gamma \overline{\varphi}  +\Big(  k^2 \beta + \frac{{4}}{\alpha} \Big)h_j \overline{\varphi}  \dd S \\
= - \int_{\partial D \setminus \Gamma} \lambda_j h_j \overline{\varphi} \dd S, \quad \forall \varphi \in V_0{(D),}
\end{multline*}
we obtain that the weak limit $h$ in addition satisfies
\[
\Delta h + k^2 h = 0 \quad  \text{in } D \quad \text{and} \quad 
 \frac{\partial h}{\partial \nu} = -\frac{1}{{4}}\Big( - \nabla_\Gamma \cdot \mu\nabla_\Gamma + k^2 \beta {+\frac{4}{\alpha}} \Big) h \quad  \text{on }\Gamma.
\]
From the assumption that $k^2$ is not an eigenvalue of \eqref{evprob1}, we conclude that $h=0$ in $D$. So, we have that $h_j$ converges weakly to $h=0$ in $V_0(D)$. Therefore, up to a subsequence, $h_j$ strongly converges to $0$ in $L^2(D)$ and $L^2(\Gamma)$. From \eqref{egf}, we obtain that up to a subsequence $\| \nabla h_j\|_{L^2(D)} \rightarrow 0$ and $\| \nabla_\Gamma h_j\|_{L^2(\Gamma)} \rightarrow 0$ as $j \rightarrow \infty$. This contradicts to the assumption \eqref{normalizedsq}.
\end{proof}

\noindent
For $k$ large enough one can show there exists at least one positive eigenvalue. To show this let us assume to the contrary that all eigenvalues $\lambda_j$ are nonpositive. From \eqref{variationh}, each eigenfunction $h_j$ corresponding to $\lambda_j$ satisfies
\begin{multline*}
\int_D | \nabla h_j|^2  - k^2 |h_j|^2\dd x + \frac{1}{{4}} \int_\Gamma \Big( k^2 \beta + \frac{{4}}{\alpha} \Big) |h_j|^2 +\mu | \nabla_\Gamma h_j|^2 \dd S 
 =  \int_{\partial D \setminus \Gamma} (-\lambda_j)|h_j|^2 \dd S.
\end{multline*}
Since the right-hand side is nonnegative for each $j$,
\[
\int_D | \nabla h_j|^2  - k^2 |h_j|^2\dd x + \frac{1}{{4}} \int_\Gamma \Big( k^2 \beta + \frac{{4}}{\alpha} \Big) |h_j|^2 +\mu | \nabla_\Gamma h_j|^2 \dd S  \geq 0.
\]
The set of  eigenfunctions  $\{h_j\}_{j\in {\mathbb N}}$  form a basis  for $V_0(D)$ since this is an eigenvalue problem for a self-adjoint  and compact operator, hence from the above we have that  
\begin{align}\label{ineq1}
\int_D | \nabla h|^2  - k^2 |h|^2\dd x + \frac{1}{{4}} \int_\Gamma \Big( k^2 \beta +\frac{{4}}{\alpha} \Big) |h|^2 +\mu | \nabla_\Gamma h|^2 \dd S \geq 0, \quad \forall h \in {V_0(D)}.
\end{align}
Let $h_0$  be a Dirichlet eigenfunction corresponding to the first Dirichlet eigenvalue $-\eta_0$ for Negative laplacian in $D$. Obviously $h_0\in V_0(D)$  and its satisfies
\begin{align*}
\int_D | \nabla h_0|^2  - \eta_0 |h_0|^2\dd x =0.
\end{align*}
Taking $h:=h_0$ in (\ref{ineq1})  we obtain 
\[
0 \leq - \int_D (k^2-\eta_0)|h_0|^2\dd x, 
\]
and if  $k^2 \geq \eta_0$, this is a contradiction.  If $k^2\beta+{4}/\alpha<0$ it is possible to show that a positive eigenvalue exists for smaller $k$ by choosing $h_0$ apropriately.

\medskip
\noindent
We close this section by giving an expression for the first eigenvalue of  \eqref{egprob}.  Let $\eta_1 := \eta_1(D, \Gamma, \tau)$ be the first eigenvalue of
\begin{eqnarray}
&\Delta u + \eta u = 0 \quad  \text{in } D&, \label{eig2}\\
& \displaystyle{\frac{\partial u}{\partial \nu} = \frac{1}{{4}}\big(  \nabla_\Gamma \cdot \mu\nabla_\Gamma \big) u \quad  \text{on }\Gamma, 
 \quad  \frac{\partial u}{\partial \nu}+ \tau u = 0} \quad  \text{on } \partial D\setminus \Gamma,  \quad  \mbox{and}\;\; u=0 \quad  \text{on } \partial \Gamma & \nonumber
\end{eqnarray}
with the additional condition $u=0$ on $\partial \Gamma$, for some $\tau >0$. Since this is an eigenvalue problem for a positive self-adjoint operator, from the Courant-Fischer inf-sup principle, we have
\begin{equation}\label{eig3}
\eta_1 = \inf_{h \in {V_0(D)}, h \neq 0} \frac{
\| \nabla u \|^2_{L^2(D)}
+ \tau \|u\|^2_{L^2(\partial D \setminus \Gamma)} + \frac{1}{{4}} {\int_\Gamma \mu |\nabla_\Gamma u |^2\dd S}}{\|u \|^2_{L^2(D)}}.
\end{equation}
This implies that for every $u \in V_0(D)$
\begin{align}\label{etaeq}
\|u \|^2_{L^2(D)} \leq \frac{1}{\eta_1} \| \nabla u \|^2_{L^2(D)}
+ \frac{\tau}{\eta_1}  \|u\|^2_{L^2(\partial D \setminus \Gamma)} +\frac{1}{{4} \eta_1} { \int_\Gamma \mu |\nabla_\Gamma u |^2 \dd S}.
\end{align}
Then, using \eqref{etaeq},  for some $\Lambda \geq 0$ we can estimate
\begin{multline}\label{lambdaest}
 \int_D  |\nabla u|^2  - k^2 |u|^2   \dd x 
+ \frac{1}{{4}}\int_\Gamma \mu |\nabla_\Gamma u|^2 +  \Big( k^2 \beta +\frac{{4}}{\alpha} \Big)|u|^2  \dd S + \Lambda \int_{\partial D \setminus \Gamma}  |u|^2\dd S \\
 \geq 
   \Big( 1- \frac{k^2}{\eta_1} \Big) \int_D  |\nabla u|^2 \dd x
+  \frac{1}{{4}}\Big( 1 - \frac{k^2}{\eta_1} \Big) \int_\Gamma \mu   |\nabla_\Gamma u|^2 \dd S
 + \frac{1}{{4}}\int_\Gamma  \Big( k^2 \beta + \frac{{4}}{\alpha} \Big)|u|^2   \dd S
\\+  \Big( \Lambda - \frac{k^2 \tau}{\eta_1} \Big)  \int_{\partial D \setminus \Gamma}  |u|^2\dd S.
\end{multline}
Choosing $\Lambda>\frac{k^2 \tau}{\eta_1}$, and  if $\inf_{\Gamma}\displaystyle{k^2 \beta + \frac{{4}}{\alpha}}>0$ and $k^2 < \eta_1$, then the left-hand side of \eqref{lambdaest} is positive with the choice of such $\Lambda$. We can write our eigenvalue problem \eqref{egprob} as in the following 
\begin{multline*}
 \int_D  \nabla u \cdot \nabla {\overline{\varphi}} - k^2 u{\overline{\varphi}}   \dd x 
+ \frac{1}{{4}}\int_\Gamma \mu \nabla_\Gamma u \cdot \nabla_{\Gamma}{\overline{\varphi}}  +  \Big( k^2 \beta + \frac{{4}}{\alpha} \Big)u {\overline{\varphi}} \dd S + \Lambda \int_{\partial D \setminus \Gamma}  u {\overline{\varphi}} \dd S
\\ = (\Lambda - \lambda) \int_{\partial D \setminus \Gamma}  u {\overline{\varphi}} \dd S, \qquad   \forall \, \varphi \in V_0(D)
\end{multline*}
Since this is an eigenvalue problem for a positive self-adjoint operator with eigenvalue parameter $\Lambda-\lambda$, we can apply the Courant-Fischer inf-sup principle to the eigenvalues $\Lambda_j := \Lambda - \lambda_j$. In particular,  we obtain 
\begin{equation}\label{eig4}
\lambda_1 = \sup_{u \in {V_0(D)}, u \neq 0}\frac{ \displaystyle{\int_D -|\nabla u|^2 + k^2 |u|^2 \dd x 
- \frac{1}{{4}} \int_\Gamma \big( k^2 \beta +\frac{{4}}{\alpha} \big) |u|^2 + \mu | \nabla_\Gamma u|^2 \dd S}}{\displaystyle{\int_{\partial D \setminus \Gamma} |u|^2 \dd S}},
\end{equation}
provided that $\inf_{\Gamma}\displaystyle{k^2 \beta + \frac{{4}}{\alpha}}>0$ and $k^2 < \eta_1$, where $\eta_1$ is defined (\ref{eig2}). Hence under these assumption the expression (\ref{eig4}) together (\ref{eig3}) shows the dependence of the first eigenvalue $\lambda_1$ on the coefficients $\alpha$, $\beta$ and $\mu$.
\section{Determination of the Eigenvalues from Far Field Data}\label{sec:deff}
In this section we show that our  target signature, i.e. the  eigenvalues of  \eqref{egprob}., can be determined from far field data. This involves a non-standard analysis of the scattering problem.  We modify the approach based on the linear sampling method  in \cite{Screens}  to our more complex problem. To this end  we can write (\ref{totalfield}) equivalently as a transmission problem: 
find $p \in V_0(D)$ and $p^s \in V(\mathbb{R}^m \setminus \overline{D})$ with $p^s+v\in V(\mathbb{R}^m \setminus \overline{D})$ (recall the definition of spaces (\ref{VD}), (\ref{VD0}) (\ref{VED}), (\ref{VED0})), such that
\begin{align}\label{transmissionprob}
\left\{\begin{aligned}\quad
&\Delta p + k^2 p = 0 \quad  \text{in } D, \\
&\Delta p^s + k^2 p^s = 0 \quad  \text{in } \mathbb{R}^m \setminus \overline{D}, \\
& (p-p^s) + \displaystyle{\frac{\alpha}{2}} \Big(\frac{\partial p}{\partial \nu} + \frac{\partial p^s}{\partial \nu}  \Big) = \varphi  \quad  \text{on } \Gamma, \\ 
&  \Big(\frac{\partial p}{\partial \nu} - \frac{\partial p^s}{\partial \nu}  \Big) + \mathcal{K}(p+p^s) = \psi    \quad  \text{on } \Gamma, \\ 
& p-p^s = \varphi_c, \quad  \text{on } \partial D \setminus \Gamma, \\
 & \frac{\partial p}{\partial \nu} - \frac{\partial p^s}{\partial \nu}  = \psi_c \quad  \text{on } \partial D \setminus \Gamma\\
 &\lim_{r \rightarrow \infty} r^{\frac{m-1}{2}} \bigg( \frac{\partial p^s}{\partial r} - ikp^s\bigg) = 0
     \end{aligned}\right.
\end{align}
where  $\varphi \in H^{-1/2}(\Gamma), \psi \in V^{-1}(\Gamma)$, $\varphi_c \in H^{1/2}(\partial D \setminus \Gamma)$ and  $\psi_c \in H^{-1/2}(\partial D \setminus \Gamma)$ be defined by 
\begin{equation}
\varphi := \Big(u^i -\frac{\alpha}{2} \displaystyle{\frac{ \partial u^i }{ \partial \nu} } \Big) \Big|_\Gamma, \quad \psi := \Big( \displaystyle{\frac{ \partial u^i }{ \partial \nu} } -\mathcal{K}u^i \Big) \Big|_\Gamma, \quad  \varphi_c := u^i|_{\partial D \setminus \Gamma}, \quad  \psi_c := \displaystyle{\frac{ \partial u^i }{ \partial \nu}}  \Big|_{\partial D \setminus \Gamma} \label{bdrycd}
\end{equation}
and  the operator $\mathcal{K}: H^1(\Gamma) \rightarrow H^{-1}(\Gamma)$ is given by
$${\mathcal K}w=\frac{1}{2}\left(-\nabla_{\Gamma}\cdot \mu\nabla_{\Gamma} w+k^2\beta w\right).$$
\noindent
Here $u^i$ in general can be any function in $V(\mathbb{R}^m \setminus \overline{D})$ or $V(D)$  with square integrable Laplacian. Now, we define the bounded linear operator $\mathcal{H} : L^2 (\mathbb{S}^{m-1}) \rightarrow H^{-1/2}(\Gamma)\times V^{-1}(\Gamma) \times H^{1/2}(\partial D \setminus \Gamma)\times H^{-1/2}(\partial D \setminus \Gamma)$ by
\begin{align*}
\mathcal{H} g :=
\Bigg( \Big(w_g^{(\lambda)}-\frac{\alpha}{2} \frac{\partial w_g^{(\lambda)} }{ \partial \nu }\Big) \Big|_\Gamma,  \Big( \frac{ \partial w_g^{(\lambda)}}{ \partial \nu} -\mathcal{K}w_g^{(\lambda)}\Big) \Big|_\Gamma, w_g^{(\lambda)}|_{\partial D \setminus \Gamma}, \frac{\partial w_g^{(\lambda)} }{ \partial \nu }\Big|_{\partial D \setminus \Gamma} \Bigg),
\end{align*}
where $w_g^{(\lambda)} $ is the total field of \eqref{auxprob} and the incident field $u^i_g$ is the Herglotz wave function defined by \eqref{hf}. From the boundary condition of \eqref{auxprob}, we have that

\begin{align*}
\mathcal{H} g =
\Big(  -\frac{\alpha}{2} \frac{\partial w_g^{(\lambda)} }{ \partial \nu } \Big|_\Gamma,   \frac{ \partial w_g^{(\lambda)}}{ \partial \nu}  \Big|_\Gamma, w_g^{(\lambda)}|_{\partial D \setminus \Gamma}, \frac{\partial w_g^{(\lambda)} }{ \partial \nu }\Big|_{\partial D \setminus \Gamma} \Big).
\end{align*}
\noindent
Let us define the bounded  compact linear operator $\mathcal{G} : \overline{\mathcal{R}(\mathcal{H})} \rightarrow L^2(\mathbb{S}^{m-1})$ by
\begin{align}\label{opg}
\mathcal{G}(\varphi, \psi, \varphi_c, \psi_c) := p_\infty^s,
\end{align}
where $p_\infty^s$ is the far-field pattern of the scattered field $p^s$ that satisfies \eqref{transmissionprob}. If we take $u^i := w_g^{(\lambda)}$ in \eqref{bdrycd}, then we obtain the factorization
\[
\mathcal{F} =  \mathcal{G}\mathcal{H},
\]
where the modified far-field operator $\mathcal{F}$ is defined by \eqref{modifiedffo}. We next  define $v\in V(D)$ by
\begin{align}\label{fnv}
v(x) := h(x) + \int_\Gamma \Big(  \zeta(y) \Phi(x,y) + \eta(y) \frac{ \partial  \Phi(x,y)}{\partial \nu_y}\Big) \dd S_y,
\end{align}
where $h\in V(D)$ satisfies the Helmholtz equation in $D$. Here, $\zeta$ and $\eta$ on $\Gamma$ are chosen such that
\begin{align}\label{zetaeta}
\left\{\begin{aligned}\quad
\zeta + \mathcal{K} \eta & = 2\mathcal{K} h \\
\eta +\frac{\alpha}{2} \zeta&  =  -\alpha \frac{\partial h}{\partial \nu}
 \end{aligned}\right.
\end{align}
\begin{assumption}\label{ass}
We assume that the operator ${\mathcal K}-\frac{2}{\alpha}{\mathcal I}:H^1_0(\Gamma)\to H^{-1}(\Gamma)$ is invertible, i.e  the following variational problem for some $f\in H^{-1}(\Gamma)$
\begin{align*}
\frac{1}{2}\int_\Gamma\left(\mu \nabla_{\Gamma} w\cdot \nabla_{\Gamma} \overline{\phi} + k^2 \beta w \overline{\phi} -\frac{{4}}{\alpha} w{\overline{\phi}} \right)\dd S=\int_{\Gamma}f \overline{\phi}\dd S, \qquad \mbox{for all}\; \phi \in H^1_0(\Gamma)
\end{align*}
has a solution $w\in H^1_0(\Gamma)$.
\end{assumption}
\noindent
The above assumption is always satisfied for $\Re(\beta)>0$ and $\alpha<0$. Otherwise we must exclude a discrete set of $k$ accumulating to $+\infty$. If Assumption \ref{ass} is satisfied than  (\ref{zetaeta}) has unique solution $\zeta, \eta\in V_0(\Gamma)$, and for such densities it is shown in \cite{H-thesis} that the single and double layer potentials in (\ref{fnv}) are in $V(D)$,  and hence $v\in V(D)$.  
\begin{remark}
Assume that $\lambda$ is an eigenvalue of \eqref{egprob} and $h^{\lambda}$ is the corresponding eigenfunction. It can be shown that $\zeta := {-2h^\lambda / \alpha - 2 \partial h^\lambda / \partial \nu }= \mathcal{K}h^{\lambda}$ and $\eta := h^{\lambda}$ satisfy \eqref{zetaeta}.
\end{remark}
\noindent
Let
\begin{multline}\label{solp}
p(x)  := h(x)  = \int_{\partial D} \Big( \ndy{v(y)} \Phi(x,y) - \ndy{\Phi(x,y)} v(y) \Big) \dd S_y \\
- \int_\Gamma  \Big(  \zeta(y) \Phi(x,y) + \eta(y) \ndy{\Phi(x,y)} \Big) \dd S_y, \quad x \in D,
\end{multline}
\begin{multline}\label{solps}
p^s(x) := 
 \int_{\partial D} \Big( \ndy{\Phi(x,y)} v(y) - \ndy{v(y)} \Phi(x,y)   \Big) \dd S_y \\
- \int_\Gamma  \Big(  \zeta(y) \Phi(x,y) + \eta(y) \ndy{\Phi(x,y)} \Big) \dd S_y, \quad x \in \mathbb{R}^m \setminus \overline{D}.
\end{multline}
From the jump relations for the single layer potential and the double layer potential \cite{mclean}, we show the following lemma.
\begin{lemma}\label{solpps}
$(p(x), p^s(x) ) $ defined by \eqref{solp}-\eqref{solps} is the solution of \eqref{transmissionprob} with 
\begin{align}\label{bcpps}
\varphi := \Big( v- \frac{\alpha}{2} \nd{v} \Big) \Big|_\Gamma, \: \:
\psi := \Big( \nd{v} - \mathcal{K}v \Big) \Big|_\Gamma, \:\:
\varphi_c := v|_{\partial D \setminus \Gamma}, \:\:
\psi_c := \nd{v} \Big|_{\partial D \setminus \Gamma},
\end{align}
where $v$ is defined by \eqref{fnv}.
\end{lemma}

\begin{lemma}
Assume that $\lambda \in \mathbb{C}$ is an eigenvalue of \eqref{egprob} and $h^{\lambda}$ is the corresponding eigenfunction. Let $w_v^{(\lambda), s}$ be the unique solution of
\begin{align}\label{solwv}
\left\{\begin{aligned}\quad
&\Delta w_v^{(\lambda), s} + k^2 w_v^{(\lambda), s} = 0 \quad  \text{in } \mathbb{R}^m \setminus \overline{D}, \\
& w_v^{(\lambda), s} = - v  \quad  \text{on } \Gamma, \\ 
&\nd{w_v^{(\lambda), s}} + \lambda w_v^{(\lambda), s} = - \nd{v} - \lambda v, \quad  \text{on } \partial D \setminus \Gamma, \\
     \end{aligned}\right.
\end{align}
with 
\begin{align}\label{srcw}
\lim_{r \rightarrow \infty} r^{\frac{m-1}{2}} \bigg( \frac{\partial w_v^{(\lambda), s} }{\partial r} - ik w_v^{(\lambda), s} \bigg) = 0,
\end{align}
where $v$ is defined by \eqref{fnv}.
Then, $\mathcal{G} (\varphi_{v}, \psi_{v}, \varphi_{v, c}, \psi_{v, c}) = 0$ with
\begin{align}\label{bcv}
\left\{\begin{aligned}\quad
\varphi_{v}& := \Big( -\frac{ \alpha}{2} \nd{(w_v^{(\lambda), s} + v)} + (w_v^{(\lambda), s} + v) \Big) \Big|_{\Gamma}, \\
\psi_{v}& := \Big( \nd{ (w_v^{(\lambda), s} + v)} - \mathcal{K} (w_v^{(\lambda), s} + v) \Big) \Big|_{\Gamma}, \\
\varphi_{v,c} &:= (w_v^{(\lambda), s} + v) |_{\partial D \setminus \Gamma}, \\
\psi_{v,c}& := \nd{(w_v^{(\lambda), s} + v)} \Big|_{\partial D \setminus \Gamma},
     \end{aligned}\right.
\end{align}
where the operator $\mathcal{G}$ is defined by \eqref{opg}.
\end{lemma}
\begin{proof}
By the definition  $\mathcal{G} (\varphi_{v}, \psi_{v}, \varphi_{v, c}, \psi_{v, c}) = p_{v, \infty}$ where $p_{v, \infty}$ is the far-field pattern of the scattered field $p_v^s$ and $(p_v, p_v^s)$ is the solution of \eqref{transmissionprob} with \eqref{bcv}.
From Lemma \ref{solpps}, $(p, p^s)$ defined by \eqref{solp}-\eqref{solps} is the solution of  \eqref{transmissionprob} with \eqref{bcpps}. Then, $p_v := p$ and $p_v^s := p^s -w_v^{(\lambda), s}$ are well-defined. From the jump relations for the single and double layer potentials and the boundary conditions \eqref{solwv}, we have that
\[
\Delta p_v^s + k^2 p_v^s = 0 \quad  \text{in } \mathbb{R}^m \setminus \overline{D}, \quad
p_v^s = 0 \quad  \text{on } \Gamma, \quad
\nd{p_v^s } + \lambda p_v^s  = 0 \quad  \text{on } \partial D \setminus \Gamma.
\]
From the uniqueness of the exterior mixed boundary value problem \cite{CakoniColtonDavid14}, $p_v^s$ must be zero in $\mathbb{R}^m \setminus \overline{D}$. Thus, we have shown that $\mathcal{G} (\varphi_{v}, \psi_{v}, \varphi_{v, 1}, \psi_{v, c}) = 0$.
\end{proof}

\begin{lemma}\label{phiinrange}
Assume that $\lambda \in \mathbb{C}$ is not an eigenvalue of \eqref{egprob}. Let $\Phi_\infty(\cdot, z)$ be the far-field pattern of the fundamental solution $\Phi(\cdot, z)$. Then, $\Phi_\infty(\cdot, z) \in \mathrm{R}(\mathcal{G})$ for any $z \in D$, where
 $\mathrm{R}(\mathcal{G})$ is the range of the operator $\mathcal{G}$.
\end{lemma}

\begin{proof}
Let $z \in D$ and $h_z \in V_0(D)$ be the unique solution of
\begin{align*}
\left\{\begin{aligned}\quad
&\Delta h_z + k^2 h_z = 0 \quad  \text{in } D, \\
&h_z + \alpha  \frac{\partial h_z}{\partial \nu} +\frac{\alpha}{2} \mathcal{K}h_z = \Phi(\cdot, z) -  \frac{\alpha}{2} \mathcal{K}\Phi(\cdot, z)   \quad  \text{on }\Gamma, \\
 & \frac{\partial h_z}{\partial \nu} + \lambda h_z = \nd{\Phi(\cdot, z)} + \lambda \Phi(\cdot, z) \quad  \text{on } \partial D \setminus \Gamma.
     \end{aligned}\right.
\end{align*}
and define $v_z$ by
\[
v_z(x) := h_z(x) + \int_\Gamma \Big( \zeta(y) \Phi(x, y) + \eta(y)\ndy{\Phi(x,y)} \Big) \dd S_y
\]
with $\zeta$ and $\eta$ given by  \eqref{zetaeta}. Then, $v_z$ satisfies the Helmholtz equation in $D$. 
\noindent 
Now, consider $(\varphi_{z}, \psi_{z}, \varphi_{z, c}, \psi_{z, c})$ defined by \eqref{bcv} with $w_v^{(\lambda), s} := w_z^{(\lambda), s}$ and $v := v_z$ where $w_z^{(\lambda), s}$ is the solution of \eqref{solwv}-\eqref{srcw}. From Lemma \ref{solpps}, $(p, p^s)$ defined by \eqref{solp}-\eqref{solps}  with $h:= h_z$ and $v := v_z$ is the solution of \eqref{transmissionprob} with the corresponding  \eqref{bcpps}. Then, $p_z := p$ and $p^s_z := p^s - w_z^{(\lambda), s}$ solve \eqref{transmissionprob} with $(\varphi_{z}, \psi_{z}, \varphi_{z, c}, \psi_{z, c})$. From the boundary condition of \eqref{solwv}, since $w_z^{(\lambda), s} = p^s - p^s_z$, we have 
\[
p_z^s = p^s + v_z \: \:  \text{on} \: \Gamma \quad \text{ and } \quad \nd{p^s_z} + \lambda p^s_z = \nd{(p^s+v_z)} + \lambda (p^s+v_z)   \: \:  \text{on} \: \partial D \setminus \Gamma.
\]
Using the definitions for $v$ and $p^s$ and the jump relations for the single and the double layer potentials, we obtain 
\[
p_z^s = \Phi(\cdot, z) \: \:  \text{on} \: \Gamma \quad \text{ and } \quad \nd{p^s_z} + \lambda p^s_z = \nd{\Phi(\cdot, z)} + \lambda\Phi(\cdot, z)  \: \:  \text{on} \: \partial D \setminus \Gamma.
\]
Since $z \in D$, $\Phi(\cdot, z)$ is well defined in $\mathbb{R}^m \setminus \overline{D}$. Therefore, $p_z^s = \Phi(\cdot, z)$ in $\mathbb{R}^m \setminus \overline{D}$. Thus, $\mathcal{G}(\varphi_{z}, \psi_{z}, \varphi_{z, c}, \psi_{z, c}) = \Phi_\infty(\cdot, z)$.  
\end{proof}

\begin{lemma}\label{nowheredense}
Assume that $\lambda \in \mathbb{C}$ is an eigenvalue of \eqref{egprob}. Then, the set 
$$S := \{ z \in D : \Phi_\infty(\cdot, z) \in \mathrm{R}(\mathcal{G}) \}$$
 is nowhere dense in $D$.
\end{lemma}

\begin{proof}
Assume to the contrary that there exists a dense subset $U$ of a ball $B$ contained in $D$ such that $z \in U \cap S$. Then, $\mathcal{G}(\varphi_{z}, \psi_{z}, \varphi_{z, c}, \psi_{z, c})= \Phi_\infty(\cdot, z)$ for some $(\varphi_{z}, \psi_{z}, \varphi_{z, c}, \psi_{z, c}) \in \overline{\mathrm{R}(\mathcal{H})}$. Thus $ (\varphi_{z}, \psi_{z}, \varphi_{z, c}, \psi_{z, c})$ satisfies  \eqref{bcv} where $w_v^{(\lambda), s} := w_z^{(\lambda), s} $ solves  \eqref{solwv}-\eqref{srcw} and $v:= v_z$ for some $v_z\in V(D)$ satisfying  $\nabla v + k^2 v = 0$ in $D$. 
Let $(p_z, p_z^s)$ be the solution of \eqref{transmissionprob} with $(\varphi_{z}, \psi_{z}, \varphi_{z, c}, \psi_{z, c})$. Since the far-field patterns of $p^s_z$ and $\Phi(\cdot, z)$ coincide, from Rellich's Lemma, we obtain that $p^s_z = \Phi(\cdot, z)$ in $\mathbb{R}^m \setminus D$. Then, $p_z$ satisfies the following:
\begin{align}\label{probpz}
\left\{\begin{aligned}\quad
&\Delta p_z + k^2 p_z = 0 \quad  \text{in } D, \\
&p_z +\alpha   \frac{\partial p_z}{\partial \nu} +\frac{ \alpha}{2} \mathcal{K}p_z =  \Phi(\cdot, z)- \frac{\alpha}{2} \mathcal{K}\Phi(\cdot, z) \quad  \text{on }\Gamma, \\
 & \frac{\partial p_z}{\partial \nu} + \lambda p_z = \nd{\Phi(\cdot, z)} + \lambda \Phi(\cdot, z) \quad  \text{on } \partial D \setminus \Gamma.
     \end{aligned}\right.
\end{align}
The above problem \eqref{probpz} is solvable if and only if for any $z \in B$,
\begin{multline}\label{solvability}
\int_\Gamma \frac{1}{\alpha} \Phi(\cdot, z) \overline{p}_z^{(\lambda)} {-} \frac{1}{2} \big( \mathcal{K}\Phi(\cdot, z)   \big) \overline{p}_z^{(\lambda) }\dd S
+ \int_{\partial D \setminus \Gamma} \nd{\Phi(\cdot, z)}  \overline{p}_z^{(\lambda)}+ \lambda \Phi(\cdot, z)  \overline{p}_z^{(\lambda)} \dd S = 0,
\end{multline}
where $\overline{p}_z^{(\lambda)} \in V_0(D)$ is an eigenfunction of \eqref{egprob}. Using the boundary condition for $\overline{p}_z^{(\lambda)}$, we can rewrite \eqref{solvability} as
\begin{align}\label{pzlambda}
\int_{\partial D} \nd{\overline{p}_z^{(\lambda)}}  \Phi(\cdot, z) - \nd{\Phi(\cdot, z)} \overline{p}_z^{(\lambda)}   \dd S
+ \int_\Gamma \nd{\Phi(\cdot, z)} \overline{p}_z^{(\lambda)}  + (\mathcal{K}  \overline{p}_z^{(\lambda)}  )\Phi(\cdot, z) \dd S = 0.
\end{align}
If we consider the left-hand side of  \eqref{pzlambda} as a function of $z$, then it solves the Helmholtz equation in $D$. Therefore, \eqref{pzlambda} holds for all $z \in D$. 
From  \eqref{pzlambda} , we have that 
\[
v^{(\lambda)}(z)  := \overline{p}_z^{(\lambda)} + \int_\Gamma \nd{\Phi(\cdot, z)} \overline{p}_z^{(\lambda)}  + (\mathcal{K}  \overline{p}_z^{(\lambda)}  )\Phi(\cdot, z) \dd S = 0.
\]
Since $ \overline{p}_z^{(\lambda)} $ is the eigenfunction corresponding to $\lambda$, $v^{(\lambda)}$ satisfies \eqref{fnv} with $\zeta :=  \mathcal{K} \overline{p}_z^{(\lambda)} $ and $\eta :=  \overline{p}_z^{(\lambda)} $. Therefore, the solution $(p, p^s)$ defined by \eqref{solp}-\eqref{solps}  with $v^{(\lambda)} =0$ is zero. Thus, $ \overline{p}_z^{(\lambda)}  = 0$, which is a contradiction.
\end{proof}

\noindent
Now, we are ready to state the main theorem that provides a criteria  to determine the eigenvalues of \eqref{egprob} from  the modified far-field equation given by
\begin{align}
\mathcal{F} g (\hat{x}) = \Phi_\infty(\hat{x}, z) \quad \text{ for } \: z \in D\label{modifiedffe}
\end{align}
where the modified far-field operator $\mathcal{F}$ is defined by \eqref{modifiedffo}. 

\begin{theorem}
\begin{enumerate}[(i)]
\item
Assume that $\lambda \in \mathbb{C}$ is not an eigenvalue of \eqref{egprob}. If $z\in D$, then there exists a sequence $\{ g_n^z\}$ in $L^2(\mathbb{S}^{m-1})$ such that
\begin{align}\label{limfg}
\lim_{n \rightarrow \infty} \| \mathcal{F} g_n^z(\hat{x}) - \Phi_\infty(\hat{x}, z) \|_{L^2(\mathbb{S}^{m-1})} = 0
\end{align}
and $\|v_{g_n^z}\|_{V(D)}$ is bounded.

\item and Assumption \ref{ass} holds. Then, for any sequence $\{ g_n^z\}$ in $L^2(\mathbb{S}^{m-1})$ satisfying \eqref{limfg}, 
 $\|v_{g_n^z}\|_{V(D)}$ cannot be bounded for any $z \in D$, except for a nowhere dense set.
\end{enumerate}
\end{theorem}

\begin{proof}
(i) Assume that $\lambda \in \mathbb{C}$ is not an eigenvalue of \eqref{egprob}. From Lemma \ref{phiinrange}, for any $z \in D$, there exists $(\varphi_{z}, \psi_{z}, \varphi_{z, c}, \psi_{z, c}) \in \overline{\mathrm{R}(\mathcal{H})}$ such that $\mathcal{G}(\varphi_{z}, \psi_{z}, \varphi_{z, c}, \psi_{z, c}) = \Phi_\infty(\cdot, z)$. Thus, there exists a sequence $\{ g_n^z\}$ in $L^2(\mathbb{S}^{m-1})$ such that 
\[
{\mathcal H}{g_n^z} := 
\Bigg( \Big(w_{g_n^z}^{(\lambda)}- \frac{\alpha}{2} \frac{\partial w_{g_n^z}^{(\lambda)} }{ \partial \nu }\Big) \Big|_\Gamma,  \Big( \frac{ \partial w_{g_n^z}^{(\lambda)}}{ \partial \nu} -\mathcal{K}w_{g_n^z}^{(\lambda)}\Big) \Big|_\Gamma, w_{g_n^z}^{(\lambda)}|_{\partial D \setminus \Gamma}, \frac{\partial w_{g_n^z}^{(\lambda)} }{ \partial \nu }\Big|_{\partial D \setminus \Gamma} \Bigg)
\]
converges to $(\varphi_{z}, \psi_{z}, \varphi_{z, c}, \psi_{z, c})$ in $H^{-1/2}(\Gamma) \times V^{-1}{(\Gamma)}\times H^{1/2}(\partial D \setminus \Gamma)\times H^{-1/2}(\partial D \setminus \Gamma)$,
where $w_{g_n^z}^{(\lambda)}$ is the total field solving \eqref{auxprob} with the incident field $u^i := v_{g_n^z}$ the Herglotz wave function defined by \eqref{hf}. Using the fact that the set of Herglotz wave functions are dense in the space of solutions to the Helmholtz equation in $V(D)$ \cite{H-thesis}, we have that  $v_{g_n^z}$ converges to $v_z \in V(D)$ such that $\Delta v_z+k^2v_z=0$ in $D$. Therefore, $\|v_{g_n^z}\|_{V}$ is bounded as $n \rightarrow \infty$ and {if $W_{g_n^z} :=\mathcal{H} g_n^z$,}
\[
\lim_{n \rightarrow \infty}  \| \mathcal{F} g_n^z(\hat{x}) - \Phi_\infty(\hat{x}, z) \|_{L^2(\mathbb{S}^{m-1})} 
= \lim_{n \rightarrow \infty}  \| \mathcal{G}(W_{g_n^z}) - \Phi_\infty(\hat{x}, z) \|_{L^2(\mathbb{S}^{m-1})} =0,
\]
since $\mathcal{G}$ is continuous.

\medskip
\noindent
(ii) Suppose that $\lambda$ in an eigenvalue of \eqref{egprob}. Assume to the contrary that there exists a sequence $\{g_n^z\}$ in $L^2(\mathbb{S}^{{m-1}})$ satisfying \eqref{limfg} such that $\| v_{g_n^z}\|_{V(D)}$ is bounded for all $z$ in a dense subset $U$ of a ball $B$ contained in $D$. Then, there exists a subsequence, still denoted by $\{v_{g_n^z}\}$, that converges weakly to a solution of the Helmholtz equation $v_z \in V(D)$. Now, we  consider
\begin{align*}
\varphi_{z} & := \Big( -\frac{\alpha}{2}  \nd{(w_z^{(\lambda), s} + v_z)} + (w_z^{(\lambda), s} + v_z) \Big) \Big|_{\Gamma}, \\
\psi_{z}&  := \Big( \nd{ (w_z^{(\lambda), s} + v_z)} - \mathcal{K} (w_z^{(\lambda), s} + v_z) \Big) \Big|_{\Gamma}, \\
\varphi_{z,c} &:= (w_z^{(\lambda), s} + v_z) |_{\partial D \setminus \Gamma}, \\
\varphi_{z,c}& := \nd{(w_z^{(\lambda), s} + v_z)} \Big|_{\partial D \setminus \Gamma},
\end{align*}
where $w_z^{(\lambda), s} $ solves \eqref{solwv}-\eqref{srcw} with $v := v_z$. We have that   $\mathcal{H}g_z^n$ converges weakly to $(\varphi_{z}, \psi_{z}, \varphi_{z, c}, \psi_{z, c})$. Since $\mathcal{G}$ is compact,  we conclude that $\mathcal{G}(\mathcal{H} g_n^z)$ converges strongly to $\mathcal{G}(\varphi_{z}, \psi_{z}, \varphi_{z, c}, \psi_{z, c})$ for all $z \in U$. From \eqref{limfg}, we have that  
$$\mathcal{G}(\varphi_{z, 1}, \varphi_{z, 2}, \psi_{z, 1}, \psi_{z, 2}) = \Phi_\infty(\cdot, z) \quad  \mbox{for all} \;\; z \in U.$$ This contradicts to Lemma \ref{nowheredense} and thus this completes the proof.
\end{proof}
\section{Numerical Results}
We next present preliminary numerical results that illustrate the detection of eigenvalues from far field data, and their sensitivity to changes in parameters.  In doing this, we make several simplifying assumptions: 1) we only perform computations in $\mathbb{R}^2$, 2) we assume that the coefficients $\alpha$, $\beta$ and $\mu$
in (\ref{totalfield}) are constant, and 3) we only consider two screens that are subsets of the unit circle.  Obviously each of these limitations should be investigated for real applications. 

The results are computed in the usual way, following for example \cite{Screens}.
For each choice of coefficients and screen $\Gamma$ we generate synthetic far field data using the finite element method with quartic polynomials on a triangular mesh that is refined slightly towards the points $\partial \Gamma$.  The domain is truncated using a radial perfectly matched layer, and curved edges are approximated by quartic polynomials.  The finite element space is discontinuous across $\partial D$ and continuity across $\partial D\setminus\Gamma$ is enforced by Nitsche's method as used in symmetric interior penalty discontinuous Galerkin methods~\cite{Arnold82}.  The code is written in Python using NGSpy~\cite{Netgen} and makes critical use of the surface
differential operators implemented in that code.

In the same way, the auxiliary problem (\ref{auxprob}) is also approximated using an NGSpy code.  Finally, in order to test  the determination of eigenvalues from far field data using a discrete version of the modified far field equation
(\ref{modifiedffe}) we also solve the eigenvalue problem (\ref{variationh}) using NGSpy.

To find eigenvalues from far field data, we discretize (\ref{modifiedffe}) by Nystrom's method using $N_{\rm{}far}$ equally spaced directions on the unit circle, and collocate the resulting linear problem.  Then we add noise to the ``measured'' far field pattern $u_\infty$. In particular if the incident and measurement directions are denoted ${d}_j$, $j=1,\cdots, N_{\rm{}far}$  then then discretized modified far field operator is represented by the $N_{\rm{}far}\times N_{\rm{}far}$ matrix $A^{(\lambda)}$ given by
\[
A^{(\lambda)}_{j,\ell}=\Delta \theta (u_{\infty}(d_j;d_\ell)-h^{(\lambda)}_\infty(d_j;d_\ell)),\quad 1\leq j,\ell\leq N_{\rm{}far},
\]
where $\Delta\theta$ is the angle between adjacent directions..  Then we compute a noisy measurement matrix using
\[
A^{(\lambda),{\rm noise}}_{j,\ell}=A^{(\lambda)}_{j,\ell}(1+\epsilon_{noise}\xi_{j,\ell} ),\quad 1\leq j,\ell\leq N_{\rm{}far},
\]
where $\epsilon_{\rm{}noise}$ is a fixed parameter and $\xi_{j,\ell}$ is a uniformly distributed random number in the interval $(-1,1)$.  In our results we choose $\epsilon_{\rm{}noise}=0.01$ which gives roughly $0.3\%$ error in the relative matrix 2-norm.
Using the noisy matrix we solve the discrete modified far field equation by Tikhonov regularization using a fixed regularization parameter $\alpha_{\rm{}Tik}=10^{-7}$
for each available $\lambda$ and $N_z=10$ source points randomly located in $D$.  We then plot the averaged $\ell_2$ norm of the discrete solution $\vec{g}$ of the modified far field equation as a function of $\lambda$.  We expect peaks in the the average norm of $\vec{g}$ to correspond to eigenvalues of $D$.

\begin{figure}
\centering
\resizebox{0.58\textwidth}{!}{\includegraphics{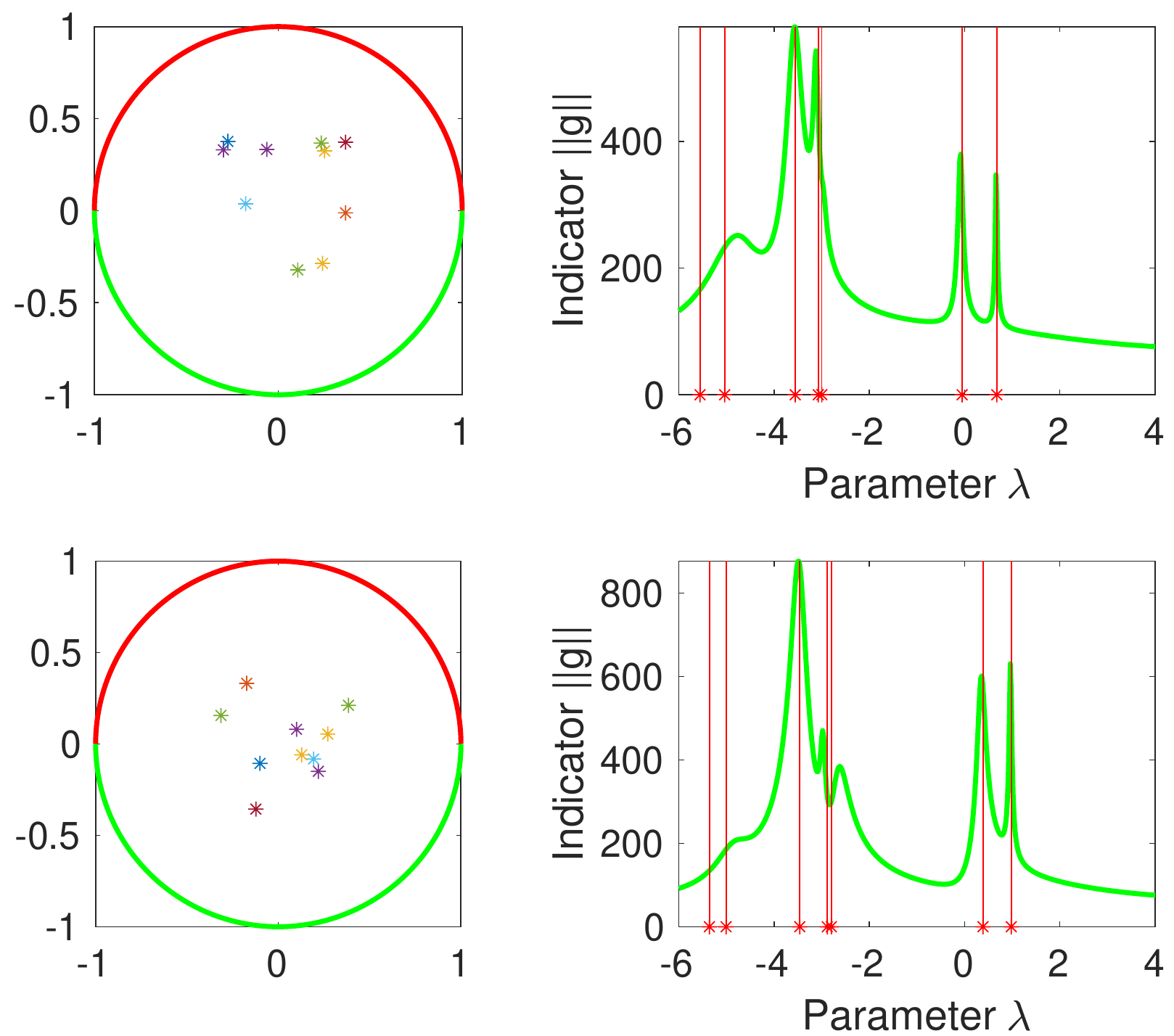}}
\caption{In the left column we show the scatterer $\Gamma$ (red curve) and the remainder of $\partial D$ as a green curve.  Asterisks show the position of the random source points  $z$ in $D$.  In the right column we show the average $\ell_2$ norm of the regularized solution of
the modified far  field equation against the eigenparameter $\lambda$. The vertical lines mark the position of the true eigenvalues
found by solving the interior eigenvalue problem. Top row: Dirichlet end condition. Bottom row: Neumann end condition. }
\label{g1_circ}
\end{figure}

In our numerical experiments we have taken $N_{\rm{}far}=120$. The minimum number of incident directions needed depends on the wave number $k$, and we have not investigated this aspect of the problem.  The two screens that we consider are the upper half of a unit circle, and a quarter of a unit circle (see Figs~\ref{g1_circ} and \ref{g2_circ} left panels).  In both cases $D$ is the unit disc, and we choose, $\alpha=-2$ and $\beta=\mu=2$ and wave number $k=4$.  For the upper half circle case, results are shown in Fig.~\ref{g1_circ} for the case of a Dirichlet boundary condition on $\partial \Gamma$. As we have mentioned in Remark~\ref{rem1}, other choices of end condition on $\partial \Gamma$ are possible and in
the lower panels of Fig.~\ref{g1_circ} we assume a homogeneous Neumann condition.  Clearly, in both cases, we can identify the largest two eigenvalues, and also information about the next three (two are close together). The corresponding result for the quarter circle scatterer (with the same parameters) is shown in Fig.~\ref{g2_circ}. From now on, we shall only present results for the Dirichlet end condition analyzed in this paper.

\begin{figure}
\centering
\resizebox{0.58\textwidth}{!}{\includegraphics{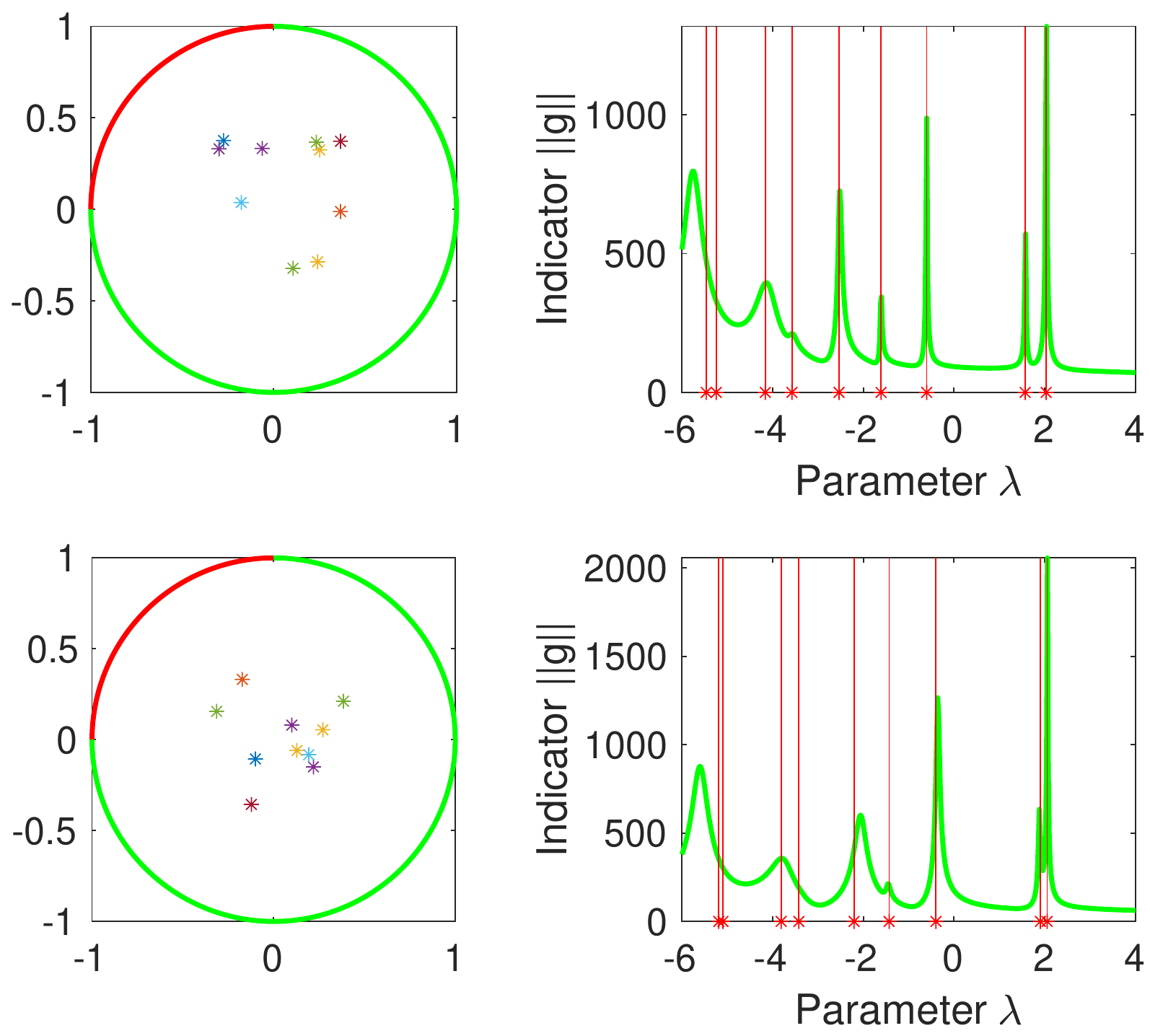}}
\caption{The layout of this figure is the same as in Fig.~\ref{g1_circ} except that the scatter is now the quarter circle shown in the left column. The same parameter values are used. Top row: Dirichlet end condition. Bottom row: Neumann end condition.}
\label{g2_circ}
\end{figure}

 In addition we also present the detection of eigenvalues when $\mu=0.2$, $\beta=1$ and $\alpha=-0.2$ to indicate that eigenvalues can be detected for quite different choices of the parameters.  These are shown in Fig.~\ref{g12_1_circ}.  It is 
apparent that for either scatterer and either choice of parameters we can detect roughly the largest 3-4 eigenvalues depending on
the end condition.  

\begin{figure}
\centering
\resizebox{0.59\textwidth}{!}{\includegraphics{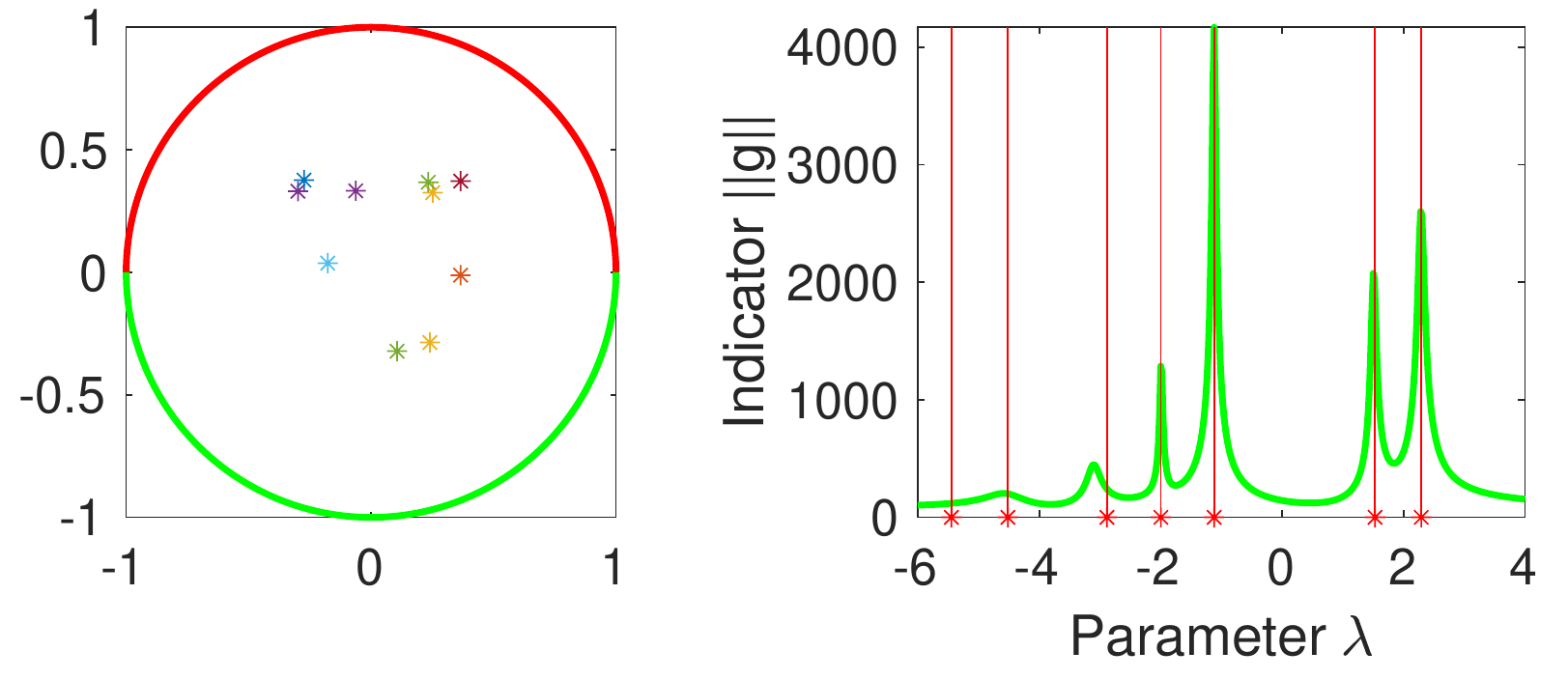}}\\
\resizebox{0.59\textwidth}{!}{\includegraphics{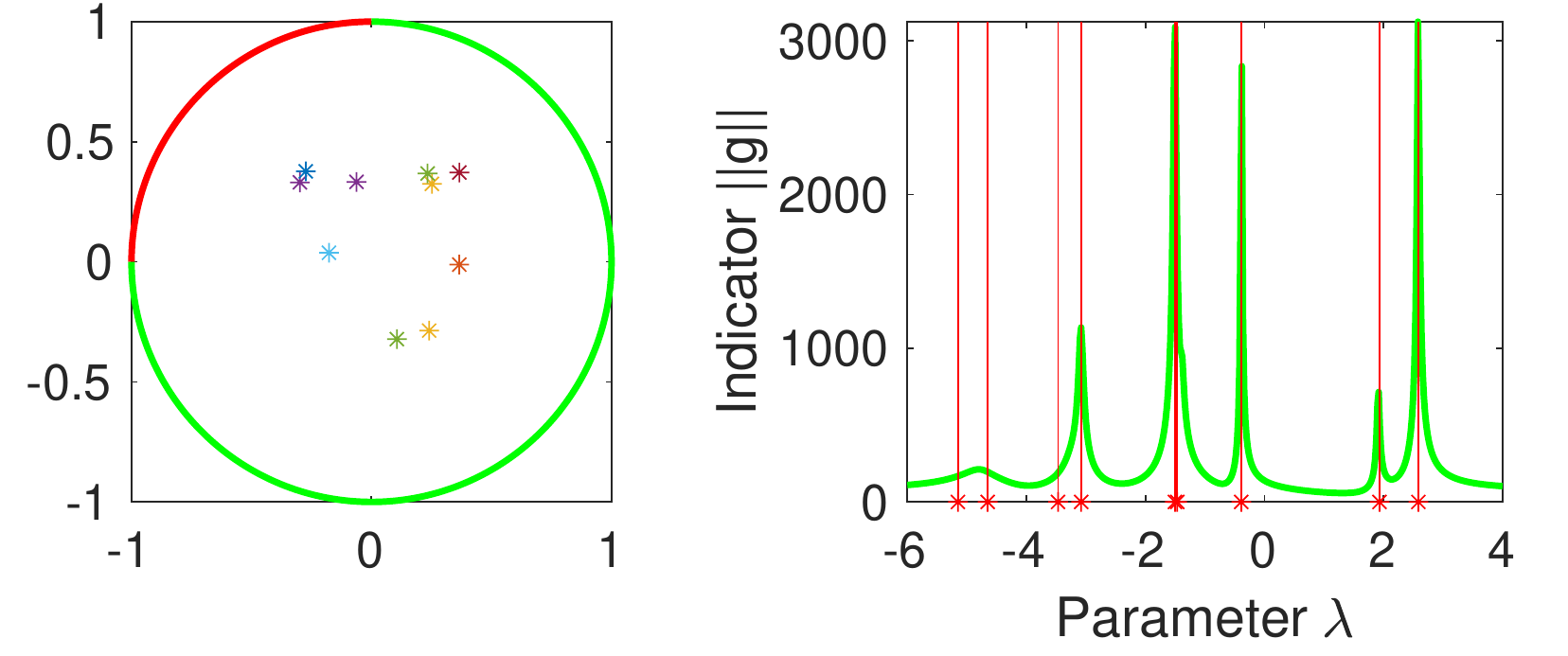}}
\caption{Here we show the detection of eigenvalues for the half and quarter circle scatterers with Dirichlet end conditions and parameters 
given by $\mu=0.2$, $\beta=1$ and $\alpha=-0.2$. See Fig.~\ref{g1_circ} for a description of the symbols used. Top row: Half circle scatterer. Bottom: Quarter circle scatterer.}
\label{g12_1_circ}
\end{figure}

The choice of the domain $D$ is, in theory, arbitrary provided it is sufficiently smooth and $\Gamma\subset \partial D$.  Of course the choice of $D$ changes the eigenvalues and eigenvectors.  For example, in Fig.~\ref{g34} we show  results of detecting eigenvalues using the parameters $\mu=\beta=2$ and  $\alpha=-2$ when $D$ is obtained by joining the end points of $\Gamma$ by
a straight line. In both cases, fewer eigenvalues can be detected and in the case of the hemisphere one eigenvalue is missed when compared to the predictions in Fig.~\ref{g1_circ}.  We have no explanation for the relatively poor performance in this case, but note that the 
solution of the auxiliary problem will have a stronger singularity at $\partial \Gamma$ compared to the case when $D$ is a circle.  We
therefore designed two new domains $D$ where arcs of circles are used to more smoothly extend $\Gamma$ to obtain $D$.  Results for these rounded domains are shown in Fig.~\ref{g67}. Eigenvalues for the hemisphere are now accurately predicted, and two eigenvalues are determined also for the quarter circle.

\begin{figure}
\centering
\resizebox{0.6\textwidth}{!}{\includegraphics{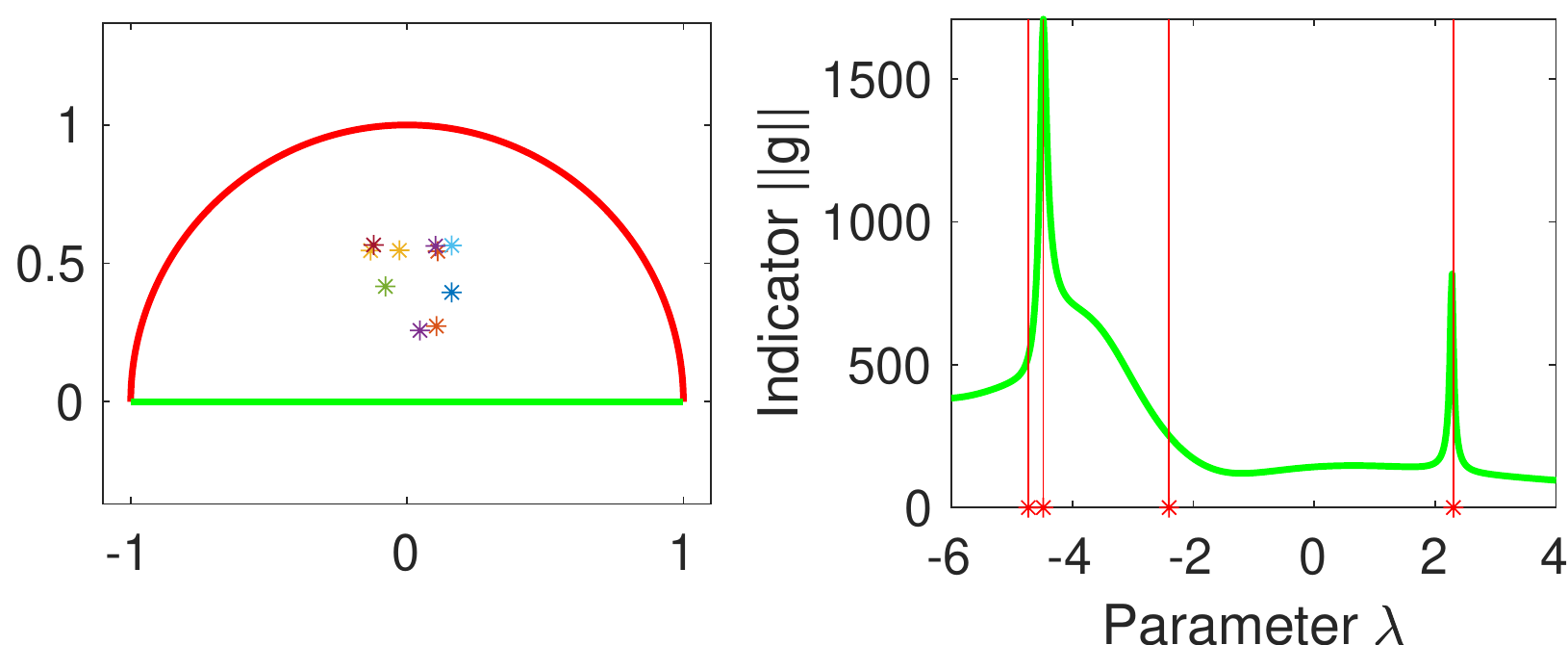}}\\
\resizebox{0.6\textwidth}{!}{\includegraphics{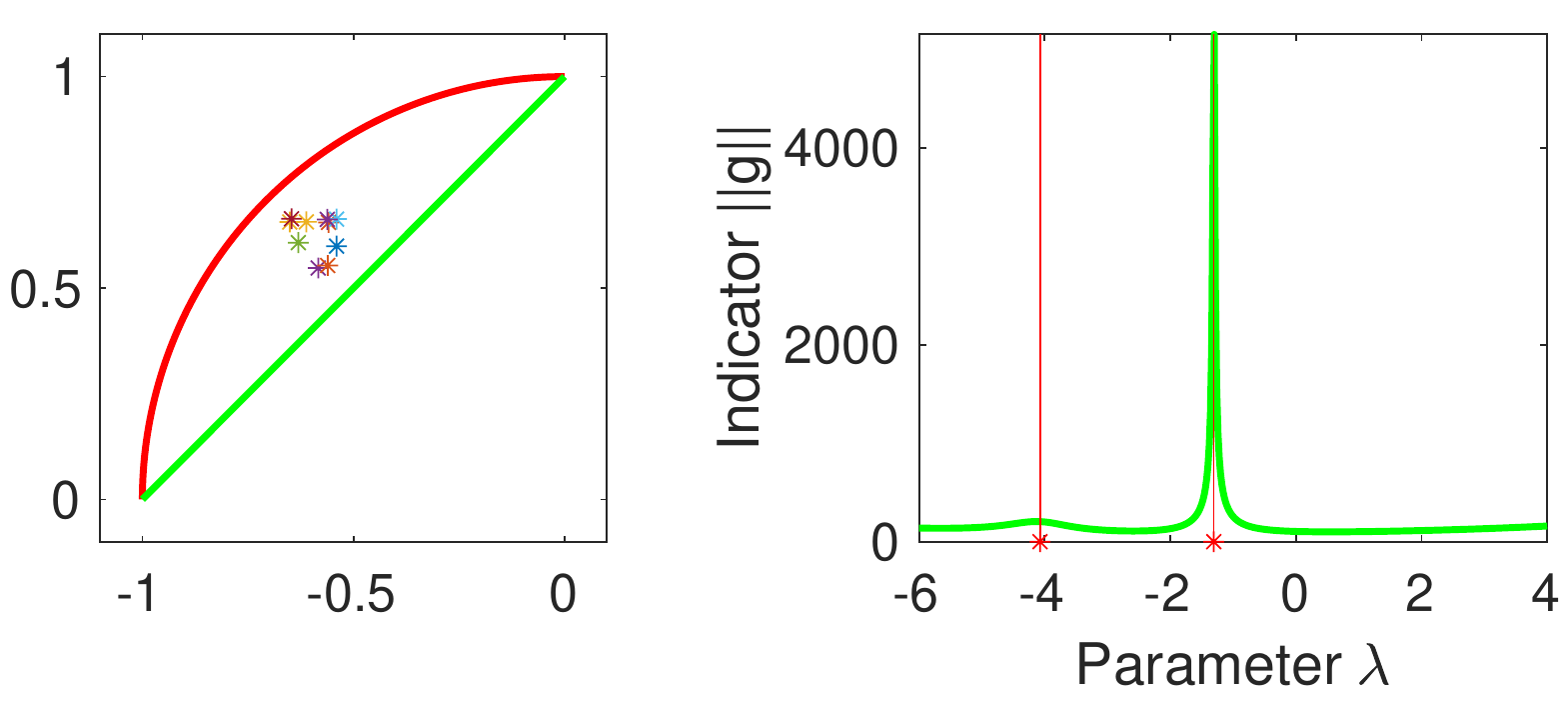}}
\caption{Here we show the detection of eigenvalues for the half and quarter circle scatterers with Dirichlet end conditions and parameters 
given by $\mu=0.2=\beta=2$ and $\alpha=-2$. The domain $D$ is now obtained by joining the end points of $\Gamma$ by a straight line.
See Fig.~\ref{g1_circ} for a description of the symbols used. Top row: Half circle scatterer. Bottom: Quarter circle scatterer.}
\label{g34}
\end{figure}

\begin{figure}
\centering
\resizebox{0.6\textwidth}{!}{\includegraphics{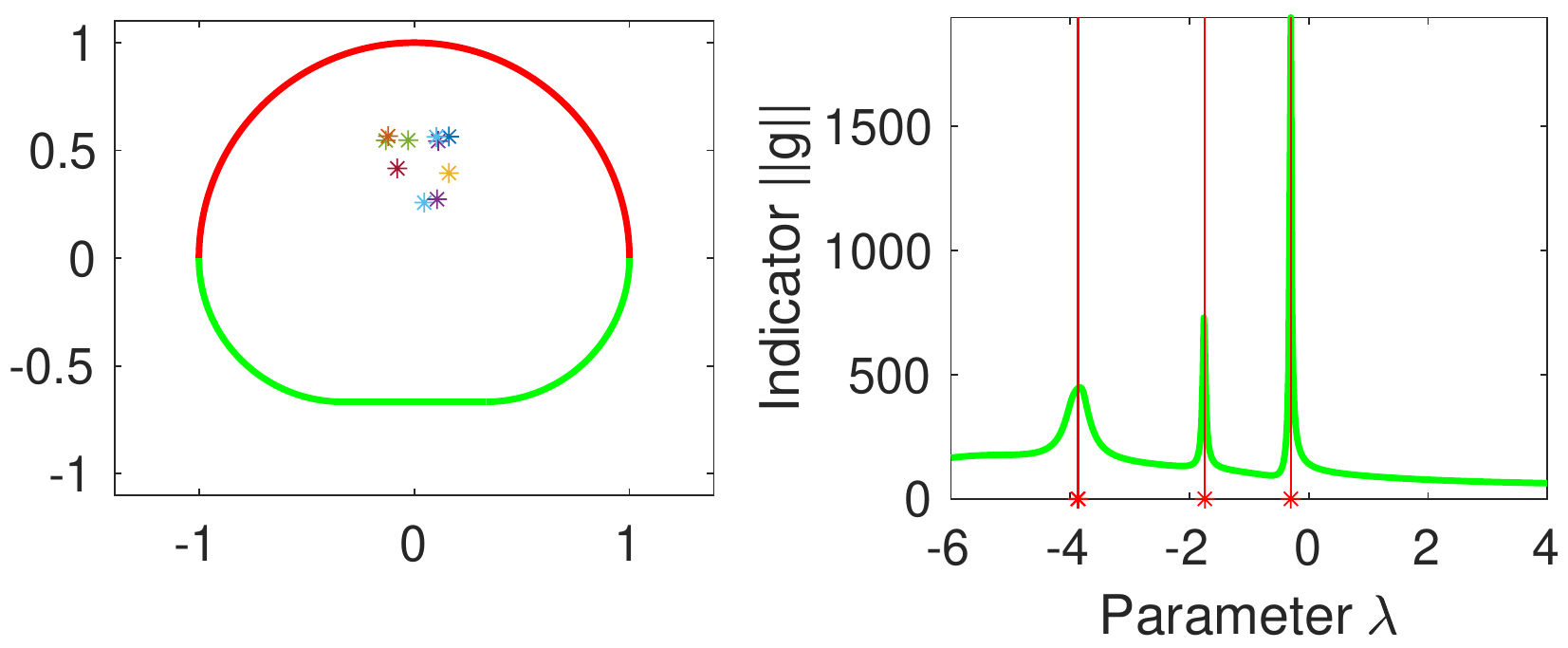}}\\
\resizebox{0.6\textwidth}{!}{\includegraphics{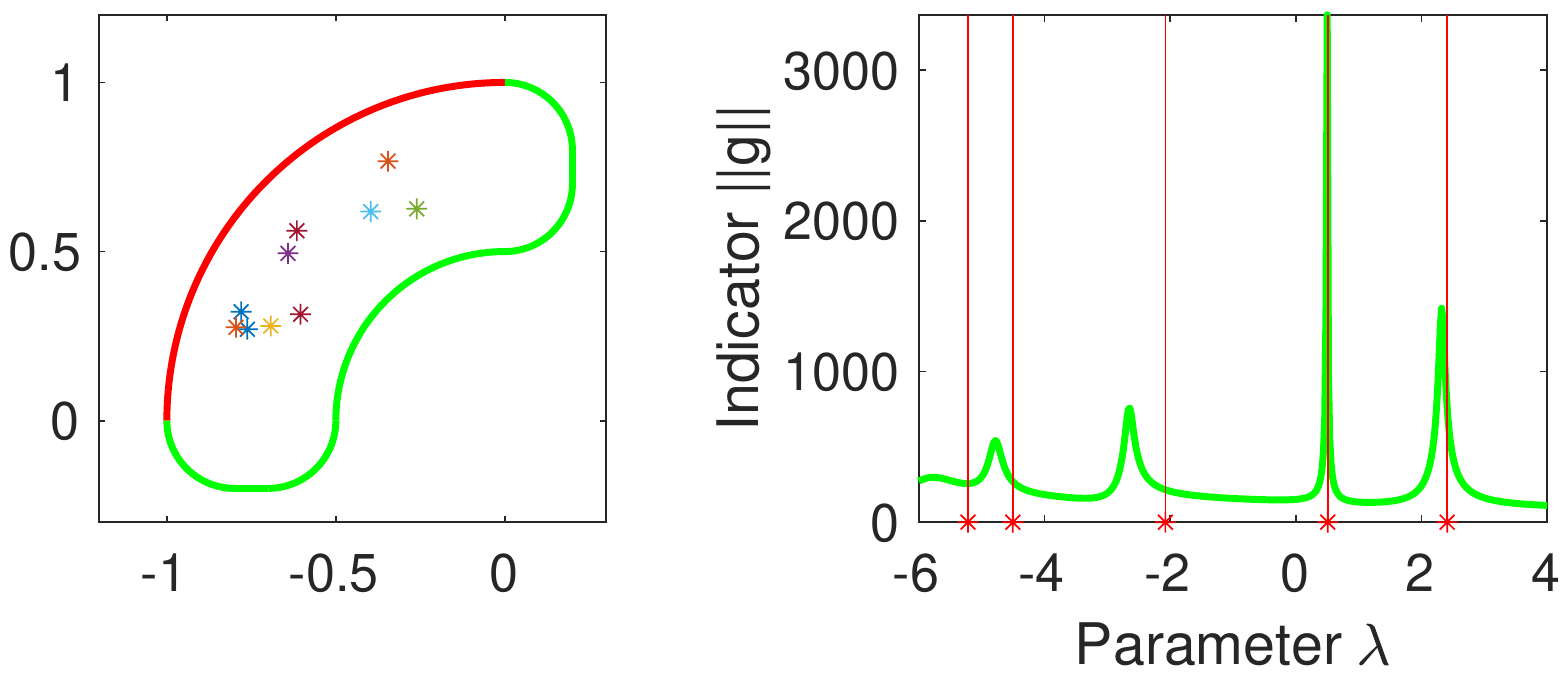}}
\caption{An example of  non-circular domains $D$ containing $\Gamma$.  These domains are smoother than those in Fig.~\ref{g34}
and allow the approximation of more eigenvalues (the same parameters are used). Top row: Half circle scatterer. Bottom: Quarter circle scatterer. }
\label{g67}
\end{figure}


Using the eigenvalue solver it is possible to examine the changes in the predicted eigenvalues of the modified far field operator as the
parameters in the surface model change.  For example, for the domain shown in Fig.~\ref{g1_circ} (a half circle scatterer with $D$ a circle),
we have examined how the first five eigenvalues in magnitude depend on $\alpha$, $\beta$ and $\mu$ in Fig.~\ref{sens_g1}.  One-by-one the parameters $\alpha$, $\beta$ and $\mu$ are varied from their base value $\alpha=-2$ and $\beta=\mu=2$.  For the parameter $\alpha$ we see that the eigenvalues sensitive to changes only for $\alpha$ greater than approximately minus one, whereas for the
other parameters the eigenvalues change throughout the range of the parameters considered. 

\begin{figure}
\centering
\resizebox{0.3\textwidth}{!}{\includegraphics{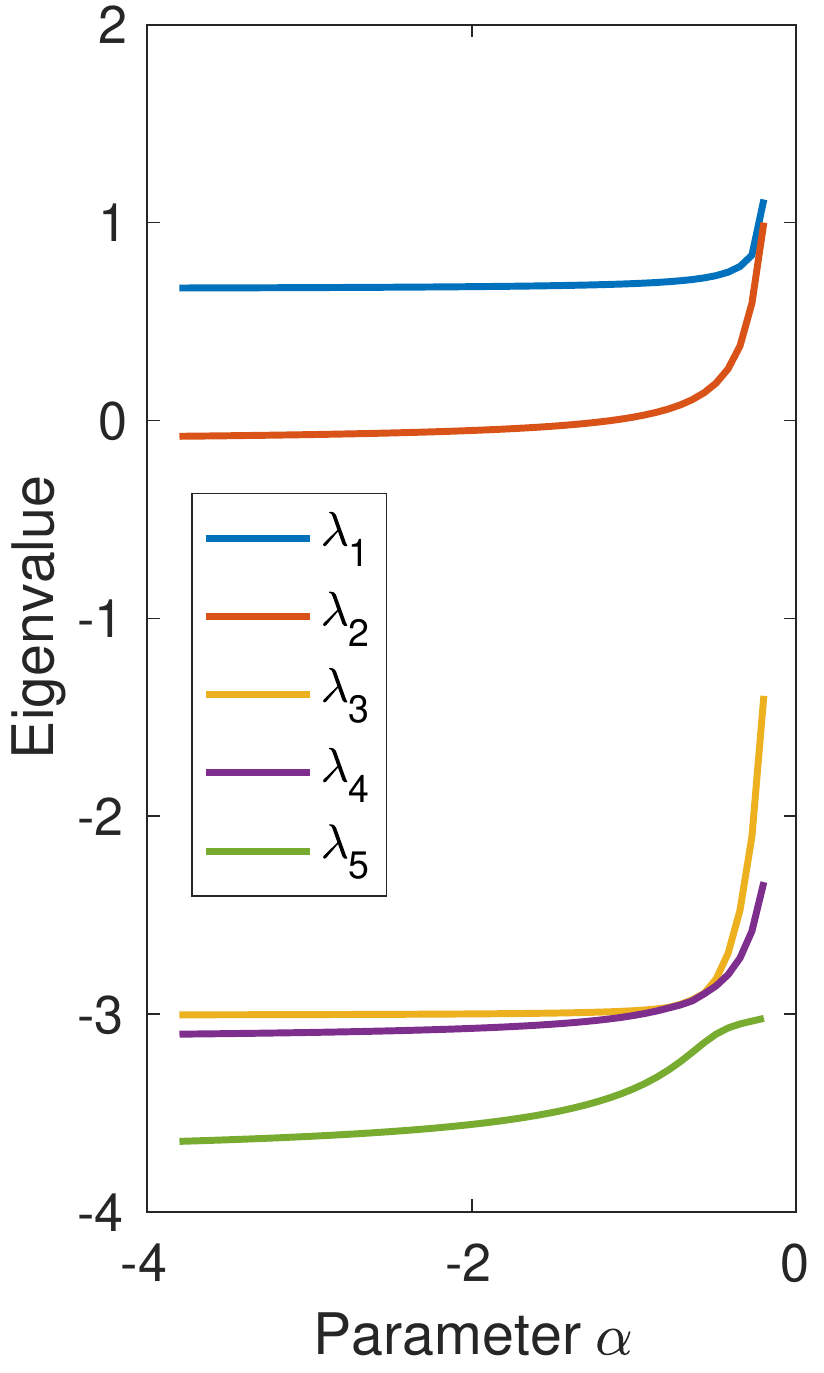}}
\resizebox{0.295\textwidth}{!}{\includegraphics{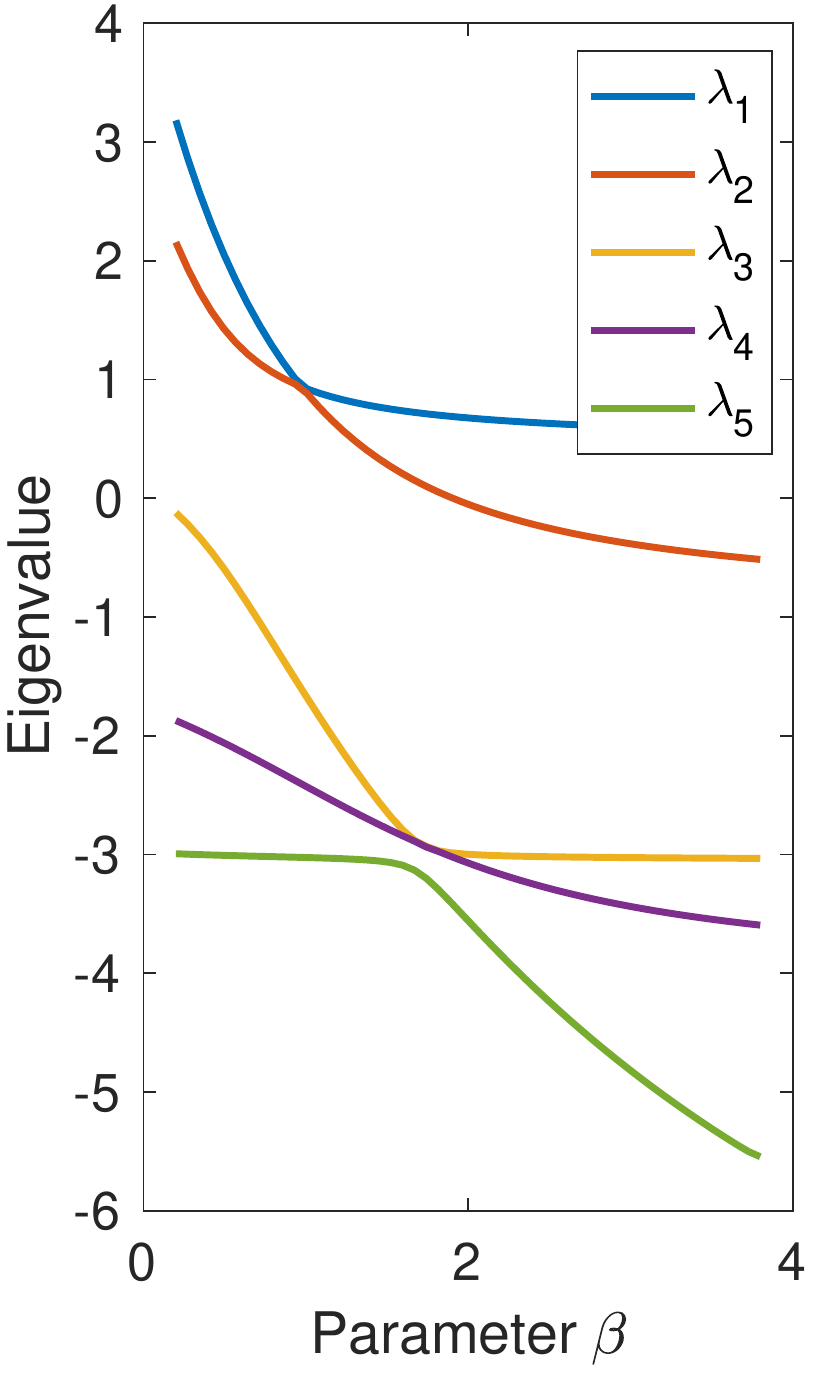}}
\resizebox{0.3 \textwidth}{!}{\includegraphics{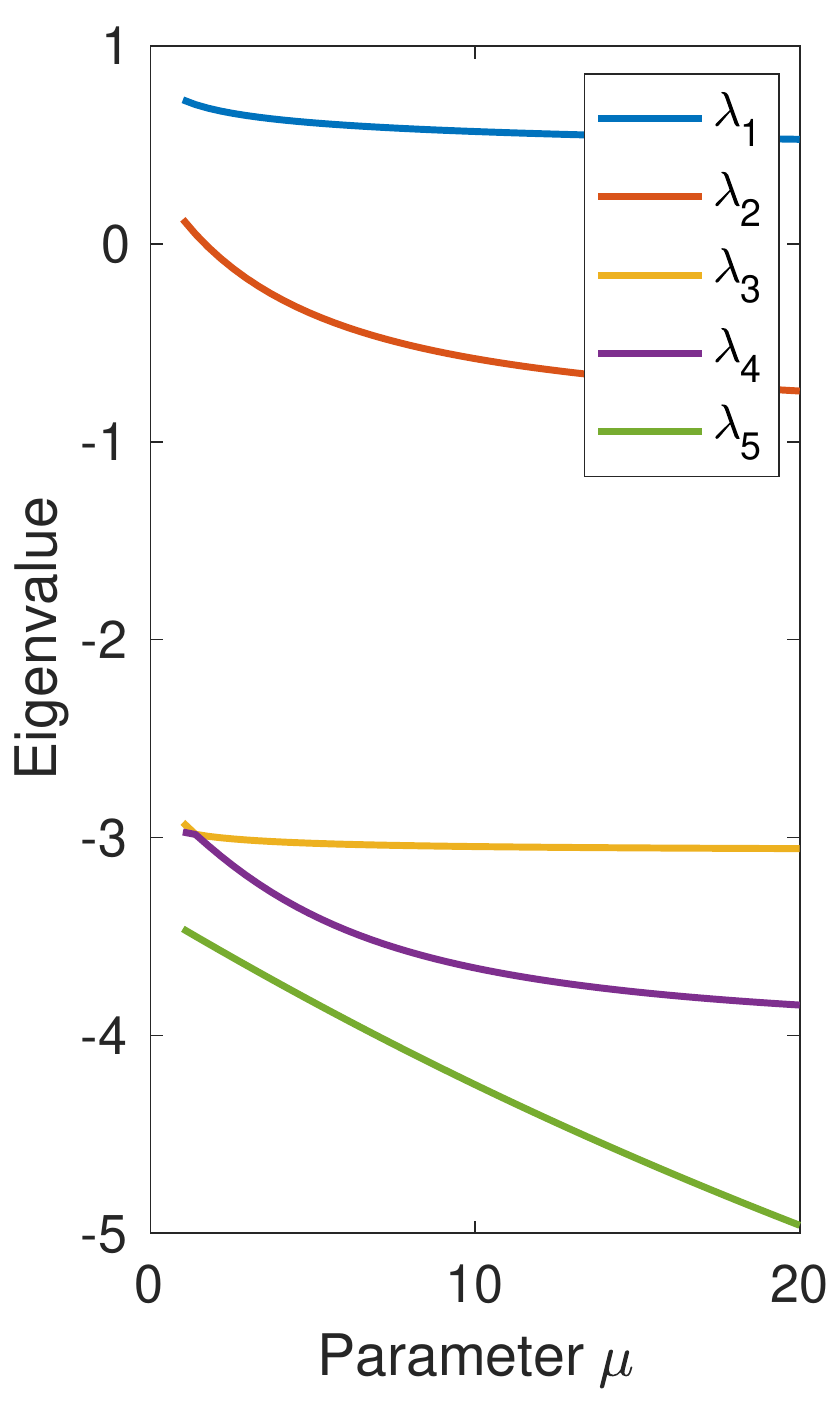}}
\caption{Changes in the first five eigenvalues (in magnitude) computed by the finite element eigenvalue solver for the half-circle scatterer and
disk $D$ as functions of the parameters $\alpha$, $\beta$ and $\mu$.  }
\label{sens_g1}
\end{figure}

The changes in the eigenvalues predicted in Fig.~\ref{sens_g1} are, of course, seen in the eigenvalues calculated via the modified far field equation. In Fig.~\ref{g1-shift} we focus on the largest pair of eigenvalues calculated by solving the modified far field equation
when $\alpha=-2$ and $\mu=2$ and $\beta=0.4,0.5,0.6$. The large change in the eigenvalues is reflected in the obvious shift in the peaks of the graphs.

\begin{figure}
\centering
\resizebox{0.6\textwidth}{!}{\includegraphics{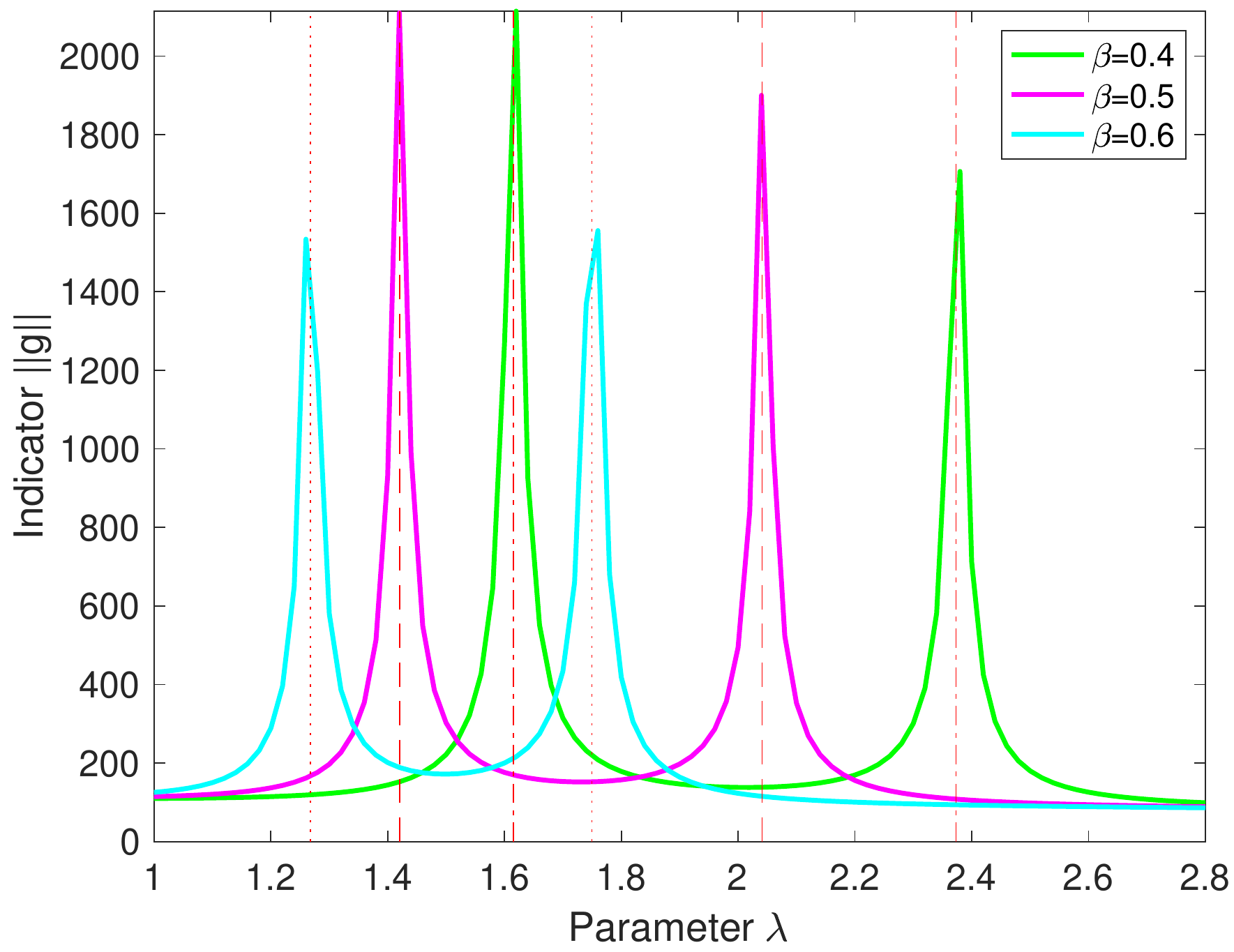}}
\caption{Predictions of the eigenvalues for the problem when $\alpha=-2$ and $\mu=2$ and $\beta=0.4,0.5$ and 0.6.  The shift in the eigenvalues
predicted in Fig.~\ref{sens_g1} (middle graph) is evident in the large translation of the peaks for the three cases.}
\label{g1-shift}
\end{figure}

\section{Conclusion}
In this paper we have examined a new set of target signatures based on eigenvalues for a thin inhomogeneity modeled by generalized transmission conditions. Concentrating on the theory, we have proved a new uniqueness result and shown that the  eigenvalues can
be determined from the solution of a modified far field operator.  Limited numerical results show that this determination can be carried out
using a discrete modified far field equation and noisy data.  More numerical testing is needed to determine how to obtain the domain $D$
that provides an accurate determination of the eigenvalues in a given case.

\section*{Acknowledgments} 
The research of F. Cakoni is partially supported  by the AFOSR Grant  FA9550-20-1-0024 and  NSF Grant DMS-2106255. The research of H. Lee is partially supported by NSF Grant DMS-2106255. The research of P. Monk is partially supported  by the AFOSR Grant  FA9550-20-1-0024.



\providecommand{\urlprefix}{URL }

\bibliographystyle{aims}
\bibliography{mybib.bib}

\medskip
\medskip

\end{document}